\documentclass[
	framed_theorems,
	framed_proofs,
	global_figure_numbers,
	print,
	abstract=true,
	english,
	tcolorbox_frames,
]{OCPreprint}

\usepackage{tikz}
\usepackage{caption}
\usepackage{subcaption}

\usepackage{enumitem}
\setlist[enumerate]{label=(\alph*)}

\renewcommand\phi\varphi
\newcommand\eps\varepsilon
\newcommand\tmpcclin{\TT_{\text{MPCC}}^{\text{lin}}}
\newcommand\tmpvclin{\TT_{\text{MPVC}}^{\text{lin}}}
\newcommand\tmpcc{\TT_{\text{MPCC}}}
\newcommand\tmpvc{\TT_{\text{MPVC}}}
\newcommand\astat[1]{\ensuremath{\text{A}_{#1}}}
\newcommand\pstat[1]{\ensuremath{\text{P}_{#1}}}
\newcommand\cstat[1]{\ensuremath{\text{P}_{#1}}}
\newcommand\tnlplin[1]{\TT_{\text{NLP}(#1)}^{\text{lin}}}
\newcommand\tnlp[1]{\TT_{\text{NLP}(#1)}}

\definecolor{todocolor}{rgb}{1.0,0.3,0.3}

\newcommand\nonpos{(-\infty,0]}
\newcommand\nonpospow[1]{(-\infty,0]^{#1}}
\newcommand\nonpossq{\nonpospow{2}}
\newcommand\nonneg{[0,\infty)}
\newcommand\nonnegpow[1]{[0,\infty)^{#1}}
\newcommand\nonnegsq{\nonnegpow{2}}
\newcommand\nset[1]{\set{1,\,\ldots,\,#1}}
\newcommand\pset{\nset{p}}

\tikzset{axis/.style={
 thin, black, -latex, shorten <=-\nudge cm, shorten >=-2*\nudge cm}}
\def\nudge{.2}
\def\feaslw{2.6pt}

\def\pastat{0}
\def\pbstat{0}
\def\aastat{0}
\def\abstat{0}
\def\castat{0}
\def\cdstat{0}
\def\weirdstat{0}
\makeatletter
\newcommand\mystore[2]{%
	\write\@auxout{\unexpanded{\global\@namedef{mystore@#1}{#2}}}%
}
\newcommand\myread[1]{%
	\ifcsname mystore@#1\endcsname
		\csname mystore@#1\endcsname%
	\else
		??%
	\fi
}
\makeatother

\newenvironment{minproblem}[2][]{%
	\begin{aligned}
		\min_{#1} \quad& #2
		\\
		\text{s.t.} \quad&
		\begin{aligned}[t]% the option t forces proper alignment
		}{%
		\end{aligned}
	\end{aligned}
}

\newcommand\labelparamtag[3]{%
	\tag{$#2(#3)$}%
	\mystore{lpt@#1}{#2}%
	\label{#1}%
}
\newcommand\paramref[2]{%
	\hyperref[#1]{\text{($\myread{lpt@#1}(#2)$)}}%
}

\ifframedtheorems
	
\else
	
\fi

\newif\ifmunkres
\munkresfalse

\let\cite\parencite

\allowdisplaybreaks

\title[Between S- and M-stationarity]{
	New stationarity conditions between strong and M-stationarity for 
	mathematical programs with complementarity constraints
}
\author{Felix Harder}
\date{September 3, 2021}
\begin{document}
%%fakesection: Title, abstract und co
\maketitle
\begin{abstract}
	We introduce new first-order necessary conditions
	for mathematical programs with complementarity constraints (MPCCs),
	which lie between strong and M-stationarity
	and have a relatively simple description.
	We show that they hold for local minimizers under 
	the rather weak constraint qualification MPCC-GCQ.
	As a generalization, we also get a class of stationarity conditions that
	lie between strong and C-stationarity and show
	that they also hold for local minimizers under MPCC-GCQ.
	We also present similar results for 
	mathematical programs with vanishing constraints (MPVCs),
	and a very simple and elementary proof of M-stationarity
	for local minimizers of MPVCs.
\end{abstract}


\begin{msc}
	\mscLink{90C33}, % MPCCs
	\mscLink{90C30} % nonlinear programming
\end{msc}

% \newpage

\section{Introduction}
\label{sec:intro}

% \subsection{MPCCs}

\begin{figure}[t]
	\centering
	\begin{subfigure}[b]{.4\textwidth}
		\centering
\begin{tikzpicture}[scale=1.1]
	\fill[blue, opacity=0.3]
	(1,0) -- (1,1) --(0,1) --(0,0) --cycle;
	\ifthenelse{\castat=1}{
		\fill[blue, opacity=0.3]
		(-1,0) -- (-1,-1) --(0,-1) --(0,0) --cycle;
	}{}
	\ifthenelse{\aastat=1}{
		\fill[blue, opacity=0.3]
		(1,0) -- (1,-1) --(0,-1) --(0,0) --cycle;
	}{}
	\ifthenelse{\abstat=1}{
		\fill[blue, opacity=0.3]
		(-1,0) -- (-1,1) --(0,1) --(0,0) --cycle;
	}{}
	\draw[axis] (-1,0) -- (1,0)
% 	node[right=2* \nudge cm] {$G(x)$\sometimes{, $\mu$}};
	node[right=2* \nudge cm] {$\bar\mu_i$};
	\draw[axis] (0,-1) -- (0,1)
% 	node[above=2*\nudge cm] {$H(x)$\sometimes{, $\nu$}};
	node[above=2*\nudge cm] {$\bar\nu_i$};
	\draw[line width=\feaslw,blue] (0,0) -- (1,0) ;
	\draw[line width=\feaslw,blue] (0,0) -- (0,1) ;
	\ifthenelse{\pastat=1 \OR \aastat=1 \OR \castat=1}{
		\draw[line width=\feaslw,blue] (0,-1) -- (0,0) ;
	}{}
	\ifthenelse{\cdstat=1}{
		\draw[line width=\feaslw,blue] (0,0) -- (-1,-1) ;
	}{}
	\ifthenelse{\weirdstat=1}{
		\draw[line width=0.7*\feaslw,smooth,blue,domain=0:1,variable=\x ] 
% 		plot ({-sin(300*\x)*sin(300*\x)}, {-\x});
		plot ({-sin(5*deg(\x))*sin(5*deg(\x))}, {-\x});
	}{}
	\ifthenelse{\pbstat=1 \OR \abstat=1 \OR \castat=1}{
		\draw[line width=\feaslw,blue] (-1,0) -- (0,0) ;
	}{}
\end{tikzpicture}
\def\pastat{0}
\def\pbstat{0}
\def\aastat{0}
\def\abstat{0}
\def\castat{0}
\def\cdstat{0}
		\caption{strong stationarity}
		\label{item:fig:sstat}
	\end{subfigure}
	\begin{subfigure}[b]{.4\textwidth}
		\def\pastat{1}
		\def\pbstat{1}
		\centering
\begin{tikzpicture}[scale=1.1]
	\fill[blue, opacity=0.3]
	(1,0) -- (1,1) --(0,1) --(0,0) --cycle;
	\ifthenelse{\castat=1}{
		\fill[blue, opacity=0.3]
		(-1,0) -- (-1,-1) --(0,-1) --(0,0) --cycle;
	}{}
	\ifthenelse{\aastat=1}{
		\fill[blue, opacity=0.3]
		(1,0) -- (1,-1) --(0,-1) --(0,0) --cycle;
	}{}
	\ifthenelse{\abstat=1}{
		\fill[blue, opacity=0.3]
		(-1,0) -- (-1,1) --(0,1) --(0,0) --cycle;
	}{}
	\draw[axis] (-1,0) -- (1,0)
% 	node[right=2* \nudge cm] {$G(x)$\sometimes{, $\mu$}};
	node[right=2* \nudge cm] {$\bar\mu_i$};
	\draw[axis] (0,-1) -- (0,1)
% 	node[above=2*\nudge cm] {$H(x)$\sometimes{, $\nu$}};
	node[above=2*\nudge cm] {$\bar\nu_i$};
	\draw[line width=\feaslw,blue] (0,0) -- (1,0) ;
	\draw[line width=\feaslw,blue] (0,0) -- (0,1) ;
	\ifthenelse{\pastat=1 \OR \aastat=1 \OR \castat=1}{
		\draw[line width=\feaslw,blue] (0,-1) -- (0,0) ;
	}{}
	\ifthenelse{\cdstat=1}{
		\draw[line width=\feaslw,blue] (0,0) -- (-1,-1) ;
	}{}
	\ifthenelse{\weirdstat=1}{
		\draw[line width=0.7*\feaslw,smooth,blue,domain=0:1,variable=\x ] 
% 		plot ({-sin(300*\x)*sin(300*\x)}, {-\x});
		plot ({-sin(5*deg(\x))*sin(5*deg(\x))}, {-\x});
	}{}
	\ifthenelse{\pbstat=1 \OR \abstat=1 \OR \castat=1}{
		\draw[line width=\feaslw,blue] (-1,0) -- (0,0) ;
	}{}
\end{tikzpicture}
\def\pastat{0}
\def\pbstat{0}
\def\aastat{0}
\def\abstat{0}
\def\castat{0}
\def\cdstat{0}
		\caption{M-stationarity}
		\label{item:fig:mstat}
	\end{subfigure}
	\begin{subfigure}[b]{.34\textwidth}
		\def\pastat{1}
		\centering
\begin{tikzpicture}[scale=1.1]
	\fill[blue, opacity=0.3]
	(1,0) -- (1,1) --(0,1) --(0,0) --cycle;
	\ifthenelse{\castat=1}{
		\fill[blue, opacity=0.3]
		(-1,0) -- (-1,-1) --(0,-1) --(0,0) --cycle;
	}{}
	\ifthenelse{\aastat=1}{
		\fill[blue, opacity=0.3]
		(1,0) -- (1,-1) --(0,-1) --(0,0) --cycle;
	}{}
	\ifthenelse{\abstat=1}{
		\fill[blue, opacity=0.3]
		(-1,0) -- (-1,1) --(0,1) --(0,0) --cycle;
	}{}
	\draw[axis] (-1,0) -- (1,0)
% 	node[right=2* \nudge cm] {$G(x)$\sometimes{, $\mu$}};
	node[right=2* \nudge cm] {$\bar\mu_i$};
	\draw[axis] (0,-1) -- (0,1)
% 	node[above=2*\nudge cm] {$H(x)$\sometimes{, $\nu$}};
	node[above=2*\nudge cm] {$\bar\nu_i$};
	\draw[line width=\feaslw,blue] (0,0) -- (1,0) ;
	\draw[line width=\feaslw,blue] (0,0) -- (0,1) ;
	\ifthenelse{\pastat=1 \OR \aastat=1 \OR \castat=1}{
		\draw[line width=\feaslw,blue] (0,-1) -- (0,0) ;
	}{}
	\ifthenelse{\cdstat=1}{
		\draw[line width=\feaslw,blue] (0,0) -- (-1,-1) ;
	}{}
	\ifthenelse{\weirdstat=1}{
		\draw[line width=0.7*\feaslw,smooth,blue,domain=0:1,variable=\x ] 
% 		plot ({-sin(300*\x)*sin(300*\x)}, {-\x});
		plot ({-sin(5*deg(\x))*sin(5*deg(\x))}, {-\x});
	}{}
	\ifthenelse{\pbstat=1 \OR \abstat=1 \OR \castat=1}{
		\draw[line width=\feaslw,blue] (-1,0) -- (0,0) ;
	}{}
\end{tikzpicture}
\def\pastat{0}
\def\pbstat{0}
\def\aastat{0}
\def\abstat{0}
\def\castat{0}
\def\cdstat{0}
		\caption{\pstat0-stationarity}
		\label{item:fig:new_stat_a}
	\end{subfigure}
	\begin{subfigure}[b]{.34\textwidth}
		\def\pbstat{1}
		\centering
\begin{tikzpicture}[scale=1.1]
	\fill[blue, opacity=0.3]
	(1,0) -- (1,1) --(0,1) --(0,0) --cycle;
	\ifthenelse{\castat=1}{
		\fill[blue, opacity=0.3]
		(-1,0) -- (-1,-1) --(0,-1) --(0,0) --cycle;
	}{}
	\ifthenelse{\aastat=1}{
		\fill[blue, opacity=0.3]
		(1,0) -- (1,-1) --(0,-1) --(0,0) --cycle;
	}{}
	\ifthenelse{\abstat=1}{
		\fill[blue, opacity=0.3]
		(-1,0) -- (-1,1) --(0,1) --(0,0) --cycle;
	}{}
	\draw[axis] (-1,0) -- (1,0)
% 	node[right=2* \nudge cm] {$G(x)$\sometimes{, $\mu$}};
	node[right=2* \nudge cm] {$\bar\mu_i$};
	\draw[axis] (0,-1) -- (0,1)
% 	node[above=2*\nudge cm] {$H(x)$\sometimes{, $\nu$}};
	node[above=2*\nudge cm] {$\bar\nu_i$};
	\draw[line width=\feaslw,blue] (0,0) -- (1,0) ;
	\draw[line width=\feaslw,blue] (0,0) -- (0,1) ;
	\ifthenelse{\pastat=1 \OR \aastat=1 \OR \castat=1}{
		\draw[line width=\feaslw,blue] (0,-1) -- (0,0) ;
	}{}
	\ifthenelse{\cdstat=1}{
		\draw[line width=\feaslw,blue] (0,0) -- (-1,-1) ;
	}{}
	\ifthenelse{\weirdstat=1}{
		\draw[line width=0.7*\feaslw,smooth,blue,domain=0:1,variable=\x ] 
% 		plot ({-sin(300*\x)*sin(300*\x)}, {-\x});
		plot ({-sin(5*deg(\x))*sin(5*deg(\x))}, {-\x});
	}{}
	\ifthenelse{\pbstat=1 \OR \abstat=1 \OR \castat=1}{
		\draw[line width=\feaslw,blue] (-1,0) -- (0,0) ;
	}{}
\end{tikzpicture}
\def\pastat{0}
\def\pbstat{0}
\def\aastat{0}
\def\abstat{0}
\def\castat{0}
\def\cdstat{0}
		\caption{\pstat1-stationarity}
		\label{item:fig:new_stat_b}
	\end{subfigure}
	\begin{subfigure}[b]{.3\textwidth}
		\def\cdstat{1}
		\centering
\begin{tikzpicture}[scale=1.1]
	\fill[blue, opacity=0.3]
	(1,0) -- (1,1) --(0,1) --(0,0) --cycle;
	\ifthenelse{\castat=1}{
		\fill[blue, opacity=0.3]
		(-1,0) -- (-1,-1) --(0,-1) --(0,0) --cycle;
	}{}
	\ifthenelse{\aastat=1}{
		\fill[blue, opacity=0.3]
		(1,0) -- (1,-1) --(0,-1) --(0,0) --cycle;
	}{}
	\ifthenelse{\abstat=1}{
		\fill[blue, opacity=0.3]
		(-1,0) -- (-1,1) --(0,1) --(0,0) --cycle;
	}{}
	\draw[axis] (-1,0) -- (1,0)
% 	node[right=2* \nudge cm] {$G(x)$\sometimes{, $\mu$}};
	node[right=2* \nudge cm] {$\bar\mu_i$};
	\draw[axis] (0,-1) -- (0,1)
% 	node[above=2*\nudge cm] {$H(x)$\sometimes{, $\nu$}};
	node[above=2*\nudge cm] {$\bar\nu_i$};
	\draw[line width=\feaslw,blue] (0,0) -- (1,0) ;
	\draw[line width=\feaslw,blue] (0,0) -- (0,1) ;
	\ifthenelse{\pastat=1 \OR \aastat=1 \OR \castat=1}{
		\draw[line width=\feaslw,blue] (0,-1) -- (0,0) ;
	}{}
	\ifthenelse{\cdstat=1}{
		\draw[line width=\feaslw,blue] (0,0) -- (-1,-1) ;
	}{}
	\ifthenelse{\weirdstat=1}{
		\draw[line width=0.7*\feaslw,smooth,blue,domain=0:1,variable=\x ] 
% 		plot ({-sin(300*\x)*sin(300*\x)}, {-\x});
		plot ({-sin(5*deg(\x))*sin(5*deg(\x))}, {-\x});
	}{}
	\ifthenelse{\pbstat=1 \OR \abstat=1 \OR \castat=1}{
		\draw[line width=\feaslw,blue] (-1,0) -- (0,0) ;
	}{}
\end{tikzpicture}
\def\pastat{0}
\def\pbstat{0}
\def\aastat{0}
\def\abstat{0}
\def\castat{0}
\def\cdstat{0}
		\caption{\cstat{1/2}-stationarity}
		\label{item:fig:cdstat}
	\end{subfigure}
	\caption{geometric illustration of M-/S-stationarity
		and the new stationarity conditions
		for MPCCs, with $i\in\pset$
		such that $G_i(\bar x)=H_i(\bar x)=0$,
		where $\bar\mu_i$, $\bar\nu_i$
		are the Lagrange multipliers that correspond to $G_i(\bar x)\geq0$,
		$H_i(\bar x)\geq0$
	}
% 		$i\in I^{00}(\bar x)
% 	=\set{i\in\pset\mid G_i(\bar x)=H_i(\bar x)=0}$}
% , $i\in I^{00}(\bar x)$}
	\label{fig:new_stat}
\end{figure}
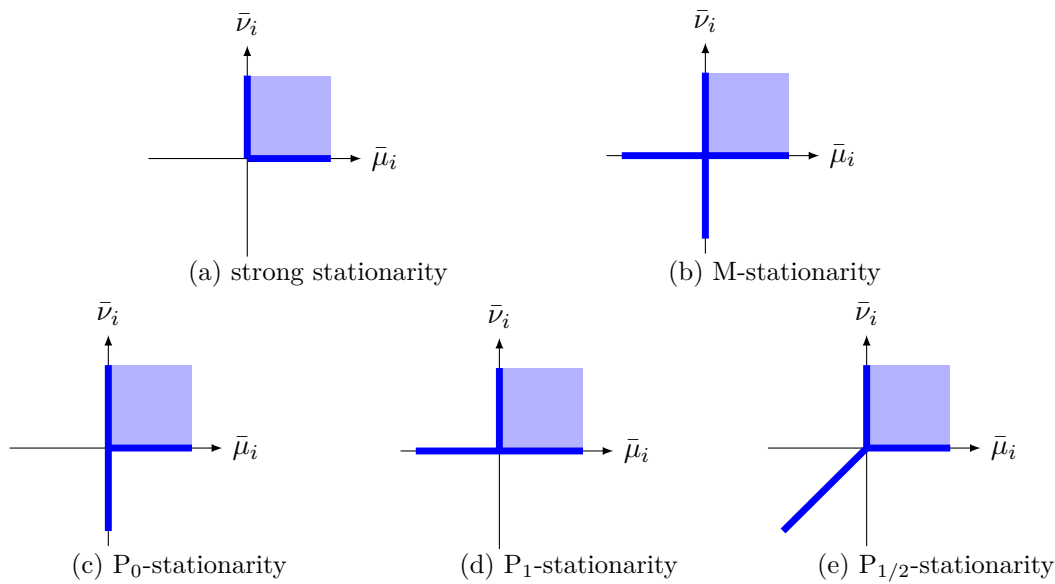

We consider mathematical programs with complementarity constraints,
or MPCCs for short,
which are nonlinear optimization problems of the form
\begin{equation*}
	\label{eq:mpcc}
	\tag{MPCC}
	\begin{minproblem}[x\in\R^n]{f(x)}
		g(x)&\leq 0,
		& h(x)&=0,
		\\
		G(x) &\geq0,
		& H(x) &\geq 0,
		& G(x)^\top H(x) &=0.
	\end{minproblem}
\end{equation*}
Here, $f:\R^n\to\R$, $g:\R^n\to\R^l$, $h:\R^n\to\R^m$,
$G,H:\R^n\to\R^p$ are differentiable functions.
For this class of problems, problem-tailored 
first-order necessary conditions
have been studied extensively in the literature.
One such stationarity condition is called strong stationarity
(see \cref{def:stat}~\ref{item:def:sstat},
illustrated in \cref{item:fig:sstat}),
which is satisfied for local minimizers under the relatively
strong constraint qualification MPCC-LICQ.
However, in \cite[Example~3]{ScheelScholtes2000} an example
is given where the data is linear but strong stationarity does not hold
for the local minimizer.

Another well-known stationarity condition is M-stationarity
(see \cref{def:stat}~\ref{item:def:mstat},
illustrated in \cref{item:fig:mstat}),
which is satisfied for local minimizers under the relatively
weak constraint qualification MPCC-GCQ
(see \cref{def:mpcc_gcq}~\ref{item:mpcc_gcq}).

The main contribution of this paper is a stationarity
condition, which we call 
\emph{P-stationarity with respect to $\alpha$} or
\emph{\pstat\alpha-stationarity},
where $\alpha\in\set{0,1}^p$ is an (a-priory given) parameter vector.
This stationarity condition is defined in
\cref{def:new_stat_mpcc}~\ref{item:def_p_stat}
and the special cases $\alpha=0$ and $\alpha=1$ are illustrated in
\cref{item:fig:new_stat_a,item:fig:new_stat_b}.
% This stationarity condition is illustrated in
% \cref{item:fig:new_stat_a,item:fig:new_stat_b}
% for the special cases $\alpha=0$ and $\alpha=1$.
To the best of our knowledge, this stationarity condition is new.
In \cref{thm:between_m_and_s} we are able to show
that \pstat\alpha-stationarity holds for local minimizers
of \eqref{eq:mpcc} under MPCC-GCQ.
The letter ``P'' stands for the Poincaré--Miranda theorem,
which is an important ingredient in the proof.
As can be seen in \cref{fig:new_stat}, this stationarity
condition lies strictly between strong and M-stationarity.
Due to the parameter $\alpha\in\set{0,1}^p$,
we actually get $2^p$ new stationarity conditions.

There are other stationarity conditions that are between
strong and M-stationarity.
One of them is linearized B-stationarity,
defined in \cref{def:stat}~\ref{item:def:linbstat}.
Others are
extended M-stationarity, strong M-stationarity,
$\QQ_M$-stationarity, and linearized M-stationarity,
see \cite{Gfrerer2014,BenkoGfrerer2017,Gfrerer2018}.
However, we are not aware of a stationarity condition
in the literature which has a nice geometrical illustration
as those in \cref{fig:new_stat} and 
which lies between strong and M-stationarity.

One important idea for the proof of \pstat\alpha-stationarity
was taken from \cite{Harder2020}:
Under MPCC-GCQ, local minimizers satisfy a system of so-called
\astat\alpha-stationarity, see \cref{prop:astat}.
Then, one can consider convex combinations of the corresponding
multipliers to obtain better stationarity conditions.
While in \cite{Harder2020} this led to a new elementary proof
of M-stationarity, in the present paper we utilize
the Poincaré--Miranda theorem to establish the stronger
\pstat\alpha-stationarity for all $\alpha\in\set{0,1}^p$.
The Poincaré--Miranda theorem
is mentioned in \cref{thm:poincare_miranda} and is a generalization
of the intermediate value theorem to higher dimensions.
It is equivalent to the Brouwer fixed-point theorem.

The idea to combine various \astat\alpha-stationarity systems
was also used for the concept of $\QQ$-stationarity,
see \cite{BenkoGfrerer2016:2,BenkoGfrerer2017}.
However, this concept does not lead directly to
\pstat\alpha-stationarity.

Our method based on the Poincaré--Miranda theorem
allows for some generalization.
We define \cstat{d}-stationarity also for the case
where $d\in [0,1]^p$ is a vector
in \cref{def:new_stat_mpcc}~\ref{item:def_cdstat}.
% Thus, we also introduce the so-called \cstat{d}-stationarity,
% where $d\in [0,1]^p$ is a vector.
This stationarity condition is illustrated in \cref{item:fig:cdstat}
for the constant vector $d=1/2$
and lies between strong stationarity and the 
so-called C-stationarity 
(see \cref{def:stat}~\ref{item:def:cstat}).
% for a definition of C-stationarity).
% It is a generalization of \pstat\alpha-stationarity
We are able to generalize \cref{thm:between_m_and_s}
and show that (for all $d\in[0,1]^p$) \cstat{d}-stationarity
holds for local minimizers under MPCC-GCQ in
\cref{thm:weird_stat_mpcc}~\ref{item:cd_stat_mpcc}.
In \cref{thm:weird_stat_mpcc}~\ref{item:weird_stat_mpcc}
we state that an even more general class of stationarity condition
holds under MPCC-GCQ.
This allows for some unusual stationarity conditions,
such as the one depicted in \cref{fig:weird_stat_mpcc}.

% \subsection{MPVCs}

We can also transfer the results from
MPCCs to 
mathematical programs with vanishing constraints,
or MPVCs for short.
These are nonlinear optimization problems of the form
\begin{equation*}
	\label{eq:mpvc}
	\tag{MPVC}
	\begin{minproblem}[x\in\R^n]{f(x)}
		g(x)&\leq 0,
		& h(x)&=0,
		\\
		H_i(x) &\geq0,
		& G_i(x) \, H_i(x) &\leq0
		\qquad\forall i\in \pset,
	\end{minproblem}
\end{equation*}
where $f$, $g$, $h$, $G$, $H$ are the same type of functions
as for MPCCs.
This problem has also been studied frequently
in the literature and was introduced
in \cite{AchtzigerKanzow2007}.
Due to some similarities in the stationarity systems,
it is relatively easy to apply the results from MPCCs also to MPVCs.
This is done in \cref{sec:between_mpvc},
where we show that the MPVC-version of
\cstat{d}-stationarity holds for all $d\in[0,1]^p$ for local minimizers
under the weak constraint qualification MPVC-GCQ
(see \cref{def:mpcc_gcq}~\ref{item:mpvc_gcq}).
% for the definition of MPVC-GCQ).
This also covers the case of \pstat\alpha-stationarity
for all $\alpha\in\set{0,1}^p$.
Again, these stationarity conditions for MPVCs are new to the best of our
knowledge.

In \cref{sec:simple_mstat}, we also provide a very short and elementary
proof of M-stationarity for local minimizers of MPVCs under MPVC-GCQ.
Although this is an already well-known stationarity condition,
we decided to include this proof as it is very short and 
the required tools are available anyways.
The proof is even simpler than the recent elementary proof
of M-stationarity for MPCCs in \cite{Harder2020} 
and relies on the observation
that \astat0-stationarity 
(defined in \cref{def:new_stat_mpvc}~\ref{item:def:astat_mpvc})
trivially implies M-stationarity in the case of MPVCs.
The proof does not rely on the Poincaré--Miranda theorem or other
complicated theory
and, as far as we know, is significantly shorter than existing proofs.
We even use a constraint qualification which is more general
than MPVC-GCQ in \eqref{eq:or_cq},
as demonstrated in \cref{ex:mpvcgcq_not_gcq_emptyset}.

As simple corollaries of our main results,
we also provide some results for the relations among the newly introduced
stationarity conditions,
see \cref{cor:relations,cor:equiv_mpcc,cor:equiv_mpvc}.

\section{Definitions}
\label{sec:definitions}

Let us make some definitions that relate
to stationarity conditions and constraint qualifications
for \eqref{eq:mpcc} and \eqref{eq:mpvc}.
% We use standard notation throughout this paper.
% It will be convenient to work with the index sets
For $x\in\R^n$ and $\alpha\in\set{0,1}^p$ we define the index sets
\begin{align*}
	I^l &:=\nset{l},\quad
	I^m := \nset{m},\quad
	I^p := \pset,
	\\
	I^g( x)&:=\set{i\in I^l\given g_i( x)=0},\\
	I^{+0}( x)&:=\set{i\in I^p\given G_i( x)>0\,\land\,H_i( x)=0},\\
	I^{0+}( x)&:=\set{i\in I^p\given G_i( x)=0\,\land\,H_i( x)>0},\\
	I^{-0}( x)&:=\set{i\in I^p\given G_i( x)<0\,\land\,H_i( x)=0},\\
	I^{-+}( x)&:=\set{i\in I^p\given G_i( x)<0\,\land\,H_i( x)>0},\\
	I^{00}( x)&:=\set{i\in I^p\given G_i( x)=0\,\land\,H_i( x)=0},\\
	I_{\alpha=0}^{00}( x)
	&:=\set{i\in I^p\given G_i( x)=0\,\land\,H_i( x)=0\,\land\,\alpha_i=0},\\
	I_{\alpha=1}^{00}( x)
	&:=\set{i\in I^p\given G_i( x)=0\,\land\,H_i( x)=0\,\land\,\alpha_i=1}.
\end{align*}
% where $x\in\R^n$ is a point and $\alpha\in\set{0,1}^p$ is a vector.
Note that if $x$ is a feasible point of \eqref{eq:mpcc},
then $I^{+0}(x)$, $I^{0+}(x)$, $I^{00}(x)$ 
form a partition of $I^p$.
Likewise, if $x$ is a feasible point of \eqref{eq:mpvc},
then $I^{+0}(x)$, $I^{-0}(x)$, $I^{-+}(x)$, $I^{0+}(x)$, $I^{00}(x)$ 
form a partition of $I^p$.
In any case, $I_{\alpha=0}^{00}(x)$ and $I_{\alpha=1}^{00}(x)$
form a partition of $I^{00}(x)$.
Let us mention that in some papers on MPVCs the notation is reversed, 
i.e.\ they use $I^{+-}(x)$ instead of $I^{-+}(x)$,
but for the sake of consistency we chose to use the same notation
for MPCCs and MPVCs.

\subsection{Constraint Qualifications}

In preparation for the definition of MPCC-GCQ and MPVC-GCQ
we introduce some cones.
\begin{definition}
	\label{def:general}
% 	Let $\bar x\in\R^n$ be a feasible point of \eqref{eq:mpcc}
% 	or \eqref{eq:mpvc}.
	\begin{enumerate}
		\item
			\label{item:tangent_cone}
			We define the \emph{tangent cone}
			at a point $\bar x\in F$ to a closed set $F\subset \R^n$ via
			\begin{equation*}
				\TT_F(\bar x):=\set*{ d\in\R^n
					\given
					\begin{aligned}
						&\exists \set{x_k}_{k\in\N}\subset F,\,
						\exists \set{t_k}_{k\in\N}\subset(0,\infty):
						\\ &\qquad
						x_k\to \bar x,\; t_k\downto 0,\;
						t_k^{-1}(x_k-\bar x)\to d
					\end{aligned}
				}.
			\end{equation*}
			If $F$ is the feasible set of \eqref{eq:mpcc},
			then we denote the tangent cone by $\tmpcc(\bar x)$.
			Likewise, if 
			$F$ is the feasible set of \eqref{eq:mpvc},
			then we denote the tangent cone by $\tmpvc(\bar x)$.
		\item
			\label{item:mpcclin}
			We define the \emph{MPCC-linearized tangent cone}
			$\tmpcclin(\bar x)\subset\R^n$ at $\bar x\in\R^n$ via
			\begin{equation*}
				\tmpcclin(\bar x):=
				\set*{ d\in\R^n
					\given
					\begin{aligned}
						\nabla g_i(\bar x)^\top d &\leq 0
						\qquad \forall i\in I^g(\bar x),
						\\
						\nabla h_i(\bar x)^\top d &= 0
						\qquad \forall i\in I^m,
						\\
						\nabla G_i(\bar x)^\top d &= 0
						\qquad \forall i\in I^{0+}(\bar x),
						\\
						\nabla H_i(\bar x)^\top d &= 0
						\qquad \forall i\in I^{+0}(\bar x),
						\\
						\nabla G_i(\bar x)^\top d &\geq 0
						\qquad \forall i\in I^{00}(\bar x),
						\\
						\nabla H_i(\bar x)^\top d &\geq 0
						\qquad \forall i\in I^{00}(\bar x),
						\\
						(\nabla G_i(\bar x)^\top d )(\nabla H_i(\bar x)^\top d)&= 0
						\qquad \forall i\in I^{00}(\bar x)
					\end{aligned}
				}.
			\end{equation*}
		\item
			We define the \emph{MPVC-linearized tangent cone}
			$\tmpvclin(\bar x)\subset\R^n$ at $\bar x\in\R^n$ via
			\begin{equation*}
				\tmpvclin(\bar x):=
				\set*{ d\in\R^n
					\given
					\begin{aligned}
						\nabla g_i(\bar x)^\top d &\leq 0
						\qquad \forall i\in I^g(\bar x),
						\\
						\nabla h_i(\bar x)^\top d &= 0
						\qquad \forall i\in I^m,
						\\
						\nabla G_i(\bar x)^\top d &\leq 0
						\qquad \forall i\in I^{0+}(\bar x),
						\\
						\nabla H_i(\bar x)^\top d &= 0
						\qquad \forall i\in I^{+0}(\bar x),
						\\
						\nabla H_i(\bar x)^\top d &\geq 0
						\qquad \forall i\in I^{00}(\bar x)\cup I^{-0}(\bar x),
						\\
						(\nabla G_i(\bar x)^\top d )
						(\nabla H_i(\bar x)^\top d) &\leq 0
						\qquad \forall i\in I^{00}(\bar x)
					\end{aligned}
				}.
			\end{equation*}
	\end{enumerate}
\end{definition}
Note that in many instances of \eqref{eq:mpcc} or \eqref{eq:mpvc},
% $\tmpcc(\bar x)$ and $\tmpcclin(\bar x)$ are
these cones are nonconvex sets.
Recall that the polar cone $C\polar$ of a set $C\subset\R^n$
is defined via
\begin{equation*}
	C\polar:=\set{d\in\R^n\given d^\top y\leq0 \quad\forall y\in C}.
\end{equation*}
Now we are ready to give the definition of
MPCC-GCQ and MPVC-GCQ.
% and its -ACQ variants.
%
\begin{definition}
	\label{def:mpcc_gcq}
	\begin{enumerate}
		\item
			\label{item:mpcc_gcq}
			Let $\bar x\in\R^n$ be a feasible point of \eqref{eq:mpcc}.
			We say that $\bar x$ satisfies the
			\emph{MPCC-tailored Guignard constraint qualification},
			or \emph{MPCC-GCQ}, if
			\begin{equation*}
				\tmpcc(\bar x)\polar = \tmpcclin(\bar x)\polar
			\end{equation*}
			holds.
			If $\tmpcc(\bar x)=\tmpcclin(\bar x)$ holds
			then we say that $\bar x$ satisfies \emph{MPCC-ACQ}.
		\item
			\label{item:mpvc_gcq}
			Let $\bar x\in\R^n$ be a feasible point of \eqref{eq:mpvc}.
			We say that $\bar x$ satisfies the
			\emph{MPVC-tailored Guignard constraint qualification},
			or \emph{MPVC-GCQ}, if
			\begin{equation*}
				\tmpvc(\bar x)\polar = \tmpvclin(\bar x)\polar
			\end{equation*}
			holds.
			If $\tmpvc(\bar x)=\tmpvclin(\bar x)$ holds
			then we say that $\bar x$ satisfies \emph{MPVC-ACQ}.
	\end{enumerate}
\end{definition}
The definition for MPCC-GCQ can also be found in
\cite[(41)]{FlegelKanzowOutrata2006},
where it is called MPEC-GCQ.
The definition of MPVC-GCQ and MPVC-ACQ can also be found in
\cite[Definition~2.8]{HoheiselKanzow2008}.

Clearly, MPCC-ACQ implies MPCC-GCQ.
% and MPVC-ACQ implies MPVC-GCQ.
% Similar to ordinary nonlinear optimization problems,
We mention that there are also other stronger constraint qualifications
(such as MPCC-MFCQ if $g$, $h$, $G$, $H$ are continuously differentiable)
which imply MPCC-ACQ or MPCC-GCQ and are easier to verify,
see e.g.\ \cite[Theorem~3.2]{Ye2005}.
In particular, we emphasize that MPCC-GCQ and MPCC-ACQ
are satisfied at every feasible point of \eqref{eq:mpcc}
if the functions $g$, $h$, $G$, $H$ are affine.
The same statements are true for MPVCs.

%
% \begin{remark}
% 	\label{rem:affine}
% 	Suppose that the functions $g$, $h$, $G$, $H$ are affine linear.
% 	Then every feasible point of \eqref{eq:mpcc} satisfies
% 	MPCC-GCQ and MPCC-ACQ.
% \end{remark}
% A proof of this statement can be found in
% \cite[Theorem~3.2]{FlegelKanzow2005:3}.

\subsection{Stationarity conditions}

We continue with the definition of tailored stationarity systems.
\begin{definition}
	\label{def:stat}
	Let $\bar x\in\R^n$ be a feasible point of \eqref{eq:mpcc}.
	\begin{enumerate}
		\item
			\label{item:def:wstat}
			We call $\bar x$ a \emph{weakly stationary} or 
			\emph{W-stationary} point of \eqref{eq:mpcc}
			if there exist multipliers
			$\bar\lambda\in\R^l$, $\bar\eta\in\R^m$,
			$\bar\mu,\bar\nu\in\R^p$ with
			\begin{subequations}
% 				\label{eq:wstat_mpcc}
				\begin{align*}
% 					\label{eq:wstat_x}
					&&&&\mathllap{
						\nabla f(\bar x) 
						+ \sum_{i\in I^l} \bar\lambda_i\nabla g_i(\bar x)
						+ \sum_{i\in I^m} \bar\eta_i\nabla h_i(\bar x)
						- \sum_{i\in I^p} \bar\mu_i\nabla G_i(\bar x)
						- \sum_{i\in I^p} \bar\nu_i\nabla H_i(\bar x)
					}
					&=0, \\
					&\hspace{18em}&&\forall i\in I^g(\bar x):
					&
					\bar\lambda_i &\geq0,
					\\
					&&&\forall i\in I^l\setminus I^g(\bar x):
					&
					\bar\lambda_i &=0,
					\\
					&&&\forall i\in I^{+0}(\bar x):
					&
					\bar\mu_i  &=0,
					\\
% 					\label{eq:wstat_0+}
					&&&\forall i\in I^{0+}(\bar x):
					&
					\bar\nu_i  &=0.
				\end{align*}
			\end{subequations}
		\item
			We call $\bar x$ an \emph{A-stationary} point if
			it is weakly stationary and the multipliers $\bar\mu,\bar\nu$
			satisfy the additional condition
			$\bar\mu_i\geq0\lor\bar\nu_i\geq0$
			for all $i\in I^{00}(\bar x)$.
		\item
			\label{item:def:cstat}
			We call $\bar x$ a \emph{C-stationary} point if
			it is weakly stationary and the multipliers $\bar\mu,\bar\nu$
			satisfy the additional condition
			$\bar\mu_i\bar\nu_i\geq0$
			for all $i\in I^{00}(\bar x)$.
		\item
			\label{item:def:mstat}
			We call $\bar x$ an \emph{M-stationary} point if
			it is weakly stationary and the multipliers $\bar\mu,\bar\nu$
			satisfy the additional condition
			$(\bar\mu_i\geq0\land \bar\nu_i\geq0)\lor \bar\mu\bar\nu=0$
			for all $i\in I^{00}(\bar x)$.
		\item
			\label{item:def:sstat}
			We call $\bar x$ a \emph{strongly stationary} or 
			\emph{S-stationary} point if
			it is weakly stationary and the multipliers $\bar\mu,\bar\nu$
			satisfy the additional condition
			$\bar\mu_i\geq0\land \bar\nu_i\geq0$
			for all $i\in I^{00}(\bar x)$.
		\item
			\label{item:def:linbstat}
			We call $\bar x$ a \emph{linearized B-stationary} point
			if 
			$-\nabla f(\bar x)\in \tmpcclin(\bar x)\polar$
			holds.
	\end{enumerate}
\end{definition}
Parts~\ref{item:def:wstat} to \ref{item:def:sstat}
of this definition can also be found in
\cite[Definitions~2.3--2.7]{Ye2005}.
The letters ``A'', ``C'' and ``M'' stand for
``alternative'', ``Clarke'' and ``Mordukhovich'', respectively.
The definition of linearized B-stationarity
appears in \cite[(25)]{FlegelKanzow2005:3}.
It is often only referred to as B-stationarity,
see, e.g., \cite[Section~2.1]{ScheelScholtes2000}.
However, since the condition $-\nabla f(\bar x)\in\tmpcc(\bar x)\polar$
(which appears in 
\cref{lem:basic_nlp_theory}~\ref{item:b_stat})
is sometimes also called B-stationarity, we prefer
the name ``linearized B-stationarity'' for
$-\nabla f(\bar x)\in\tmpcclin(\bar x)\polar$, to avoid confusion.

Let us define analogous stationarity conditions for MPVCs.
\begin{definition}
	\label{def:mpvc_stat}
	Let $\bar x\in\R^n$ be a feasible point of \eqref{eq:mpvc}.
	\begin{enumerate}
		\item
			We call $\bar x$ a \emph{weakly stationary} or 
			\emph{W-stationary} point of \eqref{eq:mpvc}
			if there exist multipliers
			$\bar\lambda\in\R^l$, $\bar\eta\in\R^m$,
			$\bar\mu,\bar\nu\in\R^p$ with
			\begin{subequations}
				\begin{align*}
% 					\label{eq:vc_wstat_x}
					&&&&\mathllap{
						\nabla f(\bar x) 
						+ \sum_{i\in I^l} \bar\lambda_i\nabla g_i(\bar x)
						+ \sum_{i\in I^m} \bar\eta_i\nabla h_i(\bar x)
						+ \sum_{i\in I^p} \bar\mu_i\nabla G_i(\bar x)
						- \sum_{i\in I^p} \bar\nu_i\nabla H_i(\bar x)
					}
					&=0, \\
					&\hspace{9em}&&\forall i\in I^g(\bar x):
					& \bar\lambda_i &\geq0,
					\\
					&&&\forall i\in I^l\setminus I^g(\bar x):
					& \bar\lambda_i &=0,
					\\
					&&&\forall i\in 
					I^{+0}(\bar x)\cup I^{-0}(\bar x)\cup I^{-+}(\bar x):
					& \bar\mu_i  &=0,
					\\
					&&&\forall i\in I^{0+}(\bar x)\cup I^{00}(\bar x):
					& \bar\mu_i  &\geq0,
					\\
					&&&\forall i\in I^{0+}(\bar x)\cup I^{-+}(\bar x):
					& \bar\nu_i  &=0,
					\\
% 					\label{eq:vc_wstat_0+}
					&&&\forall i\in I^{-0}(\bar x):
					& \bar\nu_i  &\geq0.
				\end{align*}
			\end{subequations}
		\item
			We call $\bar x$ an \emph{A-stationary} point if
			it is weakly stationary and the multipliers $\bar\mu,\bar\nu$
			satisfy the additional condition
			$\bar\mu_i=0\lor\bar\nu_i\geq0$
			for all $i\in I^{00}(\bar x)$.
		\item
			We call $\bar x$ a \emph{T-stationary} or \emph{C-stationary} point if
			it is weakly stationary and the multipliers $\bar\mu,\bar\nu$
			satisfy the additional condition
			$\bar\mu_i\bar\nu_i\leq0$
			for all $i\in I^{00}(\bar x)$.
		\item
			We call $\bar x$ an \emph{M-stationary} point if
			it is weakly stationary and the multipliers $\bar\mu,\bar\nu$
			satisfy the additional condition
			$\bar\mu_i\bar\nu_i=0$
			for all $i\in I^{00}(\bar x)$.
		\item
			We call $\bar x$ a \emph{strongly stationary} or 
			\emph{S-stationary} point if
			it is weakly stationary and the multipliers $\bar\mu,\bar\nu$
			satisfy the additional condition
			$\bar\mu_i=0\land \bar\nu_i\geq0$
			for all $i\in I^{00}(\bar x)$.
		\item
			We call $\bar x$ a \emph{linearized B-stationary} point
			if $-\nabla f(\bar x)\in \tmpvclin(\bar x)\polar$ holds.
	\end{enumerate}
\end{definition}
The definition of W,T,M,S-stationarity can be found in
\cite[Definition~2.3]{HoheiselKanzowSchwartz2012}.
As T-stationarity is an analogue to C-stationarity for MPCCs,
we assign C-stationarity as a synonymous name for it.
We are not aware of a previous 
definition of A-stationarity in the literature,
but we included it as an analogy to A-stationarity for MPCCs.

\begin{figure}[t]
	\centering
	\begin{subfigure}[b]{.3\textwidth}
		\def\castat{1}
		\def\aastat{1}
		\def\abstat{1}
		\centering
\begin{tikzpicture}[scale=1.1]
	\fill[blue, opacity=0.3]
	(1,0) -- (1,1) --(0,1) --(0,0) --cycle;
	\ifthenelse{\castat=1}{
		\fill[blue, opacity=0.3]
		(-1,0) -- (-1,-1) --(0,-1) --(0,0) --cycle;
	}{}
	\ifthenelse{\aastat=1}{
		\fill[blue, opacity=0.3]
		(1,0) -- (1,-1) --(0,-1) --(0,0) --cycle;
	}{}
	\ifthenelse{\abstat=1}{
		\fill[blue, opacity=0.3]
		(-1,0) -- (-1,1) --(0,1) --(0,0) --cycle;
	}{}
	\draw[axis] (-1,0) -- (1,0)
% 	node[right=2* \nudge cm] {$G(x)$\sometimes{, $\mu$}};
	node[right=2* \nudge cm] {$\bar\mu_i$};
	\draw[axis] (0,-1) -- (0,1)
% 	node[above=2*\nudge cm] {$H(x)$\sometimes{, $\nu$}};
	node[above=2*\nudge cm] {$\bar\nu_i$};
	\draw[line width=\feaslw,blue] (0,0) -- (1,0) ;
	\draw[line width=\feaslw,blue] (0,0) -- (0,1) ;
	\ifthenelse{\pastat=1 \OR \aastat=1 \OR \castat=1}{
		\draw[line width=\feaslw,blue] (0,-1) -- (0,0) ;
	}{}
	\ifthenelse{\cdstat=1}{
		\draw[line width=\feaslw,blue] (0,0) -- (-1,-1) ;
	}{}
	\ifthenelse{\weirdstat=1}{
		\draw[line width=0.7*\feaslw,smooth,blue,domain=0:1,variable=\x ] 
% 		plot ({-sin(300*\x)*sin(300*\x)}, {-\x});
		plot ({-sin(5*deg(\x))*sin(5*deg(\x))}, {-\x});
	}{}
	\ifthenelse{\pbstat=1 \OR \abstat=1 \OR \castat=1}{
		\draw[line width=\feaslw,blue] (-1,0) -- (0,0) ;
	}{}
\end{tikzpicture}
\def\pastat{0}
\def\pbstat{0}
\def\aastat{0}
\def\abstat{0}
\def\castat{0}
\def\cdstat{0}
		\caption{weak stationarity}
	\end{subfigure}
	\begin{subfigure}[b]{.3\textwidth}
		\def\aastat{1}
		\def\abstat{1}
		\centering
\begin{tikzpicture}[scale=1.1]
	\fill[blue, opacity=0.3]
	(1,0) -- (1,1) --(0,1) --(0,0) --cycle;
	\ifthenelse{\castat=1}{
		\fill[blue, opacity=0.3]
		(-1,0) -- (-1,-1) --(0,-1) --(0,0) --cycle;
	}{}
	\ifthenelse{\aastat=1}{
		\fill[blue, opacity=0.3]
		(1,0) -- (1,-1) --(0,-1) --(0,0) --cycle;
	}{}
	\ifthenelse{\abstat=1}{
		\fill[blue, opacity=0.3]
		(-1,0) -- (-1,1) --(0,1) --(0,0) --cycle;
	}{}
	\draw[axis] (-1,0) -- (1,0)
% 	node[right=2* \nudge cm] {$G(x)$\sometimes{, $\mu$}};
	node[right=2* \nudge cm] {$\bar\mu_i$};
	\draw[axis] (0,-1) -- (0,1)
% 	node[above=2*\nudge cm] {$H(x)$\sometimes{, $\nu$}};
	node[above=2*\nudge cm] {$\bar\nu_i$};
	\draw[line width=\feaslw,blue] (0,0) -- (1,0) ;
	\draw[line width=\feaslw,blue] (0,0) -- (0,1) ;
	\ifthenelse{\pastat=1 \OR \aastat=1 \OR \castat=1}{
		\draw[line width=\feaslw,blue] (0,-1) -- (0,0) ;
	}{}
	\ifthenelse{\cdstat=1}{
		\draw[line width=\feaslw,blue] (0,0) -- (-1,-1) ;
	}{}
	\ifthenelse{\weirdstat=1}{
		\draw[line width=0.7*\feaslw,smooth,blue,domain=0:1,variable=\x ] 
% 		plot ({-sin(300*\x)*sin(300*\x)}, {-\x});
		plot ({-sin(5*deg(\x))*sin(5*deg(\x))}, {-\x});
	}{}
	\ifthenelse{\pbstat=1 \OR \abstat=1 \OR \castat=1}{
		\draw[line width=\feaslw,blue] (-1,0) -- (0,0) ;
	}{}
\end{tikzpicture}
\def\pastat{0}
\def\pbstat{0}
\def\aastat{0}
\def\abstat{0}
\def\castat{0}
\def\cdstat{0}
		\caption{A-stationarity}
	\end{subfigure}
	\begin{subfigure}[b]{.3\textwidth}
		\def\castat{1}
		\centering
\begin{tikzpicture}[scale=1.1]
	\fill[blue, opacity=0.3]
	(1,0) -- (1,1) --(0,1) --(0,0) --cycle;
	\ifthenelse{\castat=1}{
		\fill[blue, opacity=0.3]
		(-1,0) -- (-1,-1) --(0,-1) --(0,0) --cycle;
	}{}
	\ifthenelse{\aastat=1}{
		\fill[blue, opacity=0.3]
		(1,0) -- (1,-1) --(0,-1) --(0,0) --cycle;
	}{}
	\ifthenelse{\abstat=1}{
		\fill[blue, opacity=0.3]
		(-1,0) -- (-1,1) --(0,1) --(0,0) --cycle;
	}{}
	\draw[axis] (-1,0) -- (1,0)
% 	node[right=2* \nudge cm] {$G(x)$\sometimes{, $\mu$}};
	node[right=2* \nudge cm] {$\bar\mu_i$};
	\draw[axis] (0,-1) -- (0,1)
% 	node[above=2*\nudge cm] {$H(x)$\sometimes{, $\nu$}};
	node[above=2*\nudge cm] {$\bar\nu_i$};
	\draw[line width=\feaslw,blue] (0,0) -- (1,0) ;
	\draw[line width=\feaslw,blue] (0,0) -- (0,1) ;
	\ifthenelse{\pastat=1 \OR \aastat=1 \OR \castat=1}{
		\draw[line width=\feaslw,blue] (0,-1) -- (0,0) ;
	}{}
	\ifthenelse{\cdstat=1}{
		\draw[line width=\feaslw,blue] (0,0) -- (-1,-1) ;
	}{}
	\ifthenelse{\weirdstat=1}{
		\draw[line width=0.7*\feaslw,smooth,blue,domain=0:1,variable=\x ] 
% 		plot ({-sin(300*\x)*sin(300*\x)}, {-\x});
		plot ({-sin(5*deg(\x))*sin(5*deg(\x))}, {-\x});
	}{}
	\ifthenelse{\pbstat=1 \OR \abstat=1 \OR \castat=1}{
		\draw[line width=\feaslw,blue] (-1,0) -- (0,0) ;
	}{}
\end{tikzpicture}
\def\pastat{0}
\def\pbstat{0}
\def\aastat{0}
\def\abstat{0}
\def\castat{0}
\def\cdstat{0}
		\caption{C-stationarity}
	\end{subfigure}
	\begin{subfigure}[b]{.3\textwidth}
		\def\aastat{1}
		\centering
\begin{tikzpicture}[scale=1.1]
	\fill[blue, opacity=0.3]
	(1,0) -- (1,1) --(0,1) --(0,0) --cycle;
	\ifthenelse{\castat=1}{
		\fill[blue, opacity=0.3]
		(-1,0) -- (-1,-1) --(0,-1) --(0,0) --cycle;
	}{}
	\ifthenelse{\aastat=1}{
		\fill[blue, opacity=0.3]
		(1,0) -- (1,-1) --(0,-1) --(0,0) --cycle;
	}{}
	\ifthenelse{\abstat=1}{
		\fill[blue, opacity=0.3]
		(-1,0) -- (-1,1) --(0,1) --(0,0) --cycle;
	}{}
	\draw[axis] (-1,0) -- (1,0)
% 	node[right=2* \nudge cm] {$G(x)$\sometimes{, $\mu$}};
	node[right=2* \nudge cm] {$\bar\mu_i$};
	\draw[axis] (0,-1) -- (0,1)
% 	node[above=2*\nudge cm] {$H(x)$\sometimes{, $\nu$}};
	node[above=2*\nudge cm] {$\bar\nu_i$};
	\draw[line width=\feaslw,blue] (0,0) -- (1,0) ;
	\draw[line width=\feaslw,blue] (0,0) -- (0,1) ;
	\ifthenelse{\pastat=1 \OR \aastat=1 \OR \castat=1}{
		\draw[line width=\feaslw,blue] (0,-1) -- (0,0) ;
	}{}
	\ifthenelse{\cdstat=1}{
		\draw[line width=\feaslw,blue] (0,0) -- (-1,-1) ;
	}{}
	\ifthenelse{\weirdstat=1}{
		\draw[line width=0.7*\feaslw,smooth,blue,domain=0:1,variable=\x ] 
% 		plot ({-sin(300*\x)*sin(300*\x)}, {-\x});
		plot ({-sin(5*deg(\x))*sin(5*deg(\x))}, {-\x});
	}{}
	\ifthenelse{\pbstat=1 \OR \abstat=1 \OR \castat=1}{
		\draw[line width=\feaslw,blue] (-1,0) -- (0,0) ;
	}{}
\end{tikzpicture}
\def\pastat{0}
\def\pbstat{0}
\def\aastat{0}
\def\abstat{0}
\def\castat{0}
\def\cdstat{0}
		\caption{\astat0-stationarity}
	\end{subfigure}
	\begin{subfigure}[b]{.3\textwidth}
		\def\abstat{1}
		\centering
\begin{tikzpicture}[scale=1.1]
	\fill[blue, opacity=0.3]
	(1,0) -- (1,1) --(0,1) --(0,0) --cycle;
	\ifthenelse{\castat=1}{
		\fill[blue, opacity=0.3]
		(-1,0) -- (-1,-1) --(0,-1) --(0,0) --cycle;
	}{}
	\ifthenelse{\aastat=1}{
		\fill[blue, opacity=0.3]
		(1,0) -- (1,-1) --(0,-1) --(0,0) --cycle;
	}{}
	\ifthenelse{\abstat=1}{
		\fill[blue, opacity=0.3]
		(-1,0) -- (-1,1) --(0,1) --(0,0) --cycle;
	}{}
	\draw[axis] (-1,0) -- (1,0)
% 	node[right=2* \nudge cm] {$G(x)$\sometimes{, $\mu$}};
	node[right=2* \nudge cm] {$\bar\mu_i$};
	\draw[axis] (0,-1) -- (0,1)
% 	node[above=2*\nudge cm] {$H(x)$\sometimes{, $\nu$}};
	node[above=2*\nudge cm] {$\bar\nu_i$};
	\draw[line width=\feaslw,blue] (0,0) -- (1,0) ;
	\draw[line width=\feaslw,blue] (0,0) -- (0,1) ;
	\ifthenelse{\pastat=1 \OR \aastat=1 \OR \castat=1}{
		\draw[line width=\feaslw,blue] (0,-1) -- (0,0) ;
	}{}
	\ifthenelse{\cdstat=1}{
		\draw[line width=\feaslw,blue] (0,0) -- (-1,-1) ;
	}{}
	\ifthenelse{\weirdstat=1}{
		\draw[line width=0.7*\feaslw,smooth,blue,domain=0:1,variable=\x ] 
% 		plot ({-sin(300*\x)*sin(300*\x)}, {-\x});
		plot ({-sin(5*deg(\x))*sin(5*deg(\x))}, {-\x});
	}{}
	\ifthenelse{\pbstat=1 \OR \abstat=1 \OR \castat=1}{
		\draw[line width=\feaslw,blue] (-1,0) -- (0,0) ;
	}{}
\end{tikzpicture}
\def\pastat{0}
\def\pbstat{0}
\def\aastat{0}
\def\abstat{0}
\def\castat{0}
\def\cdstat{0}
		\caption{\astat1-stationarity}
	\end{subfigure}
	\caption{geometric illustration of W-, A-, C-, 
		\astat\alpha-stationarity
	for MPCCs, $i\in I^{00}(\bar x)$}
	\label{fig:wac_stat_mpcc}
\end{figure}
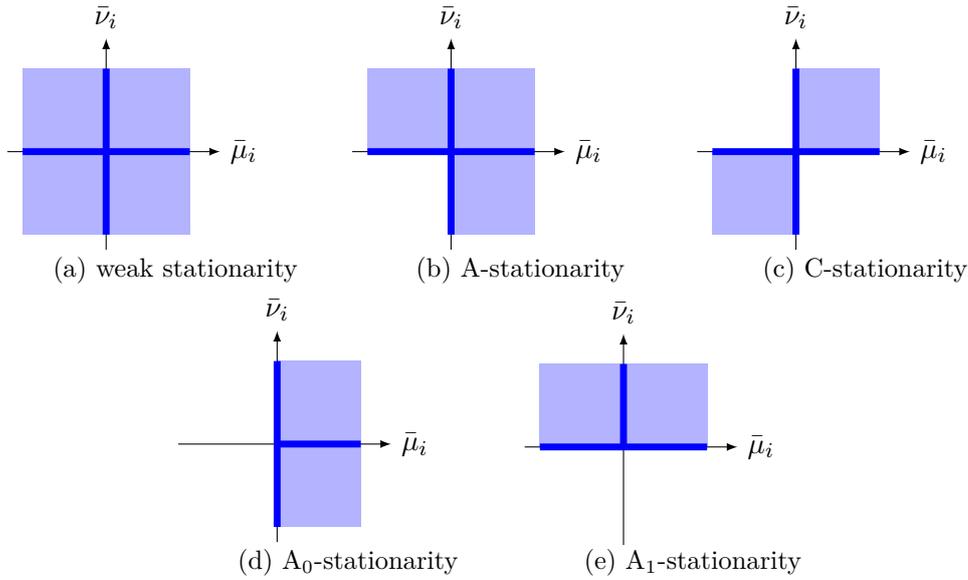

Let us also introduce some new stationarity conditions for \eqref{eq:mpcc}.
\begin{definition}
	\label{def:new_stat_mpcc}
	Let $\bar x\in\R^n$ be a feasible point of \eqref{eq:mpcc} and
	$\alpha\in\set{0,1}^p$, $d\in [0,1]^p$ be given.
	\begin{enumerate}
		\item
			\label{item:def:astat_mpcc}
			We call $\bar x$ \emph{\astat\alpha-stationary}
			if it is weakly stationary and the multipliers $\bar\mu$, $\bar\nu$
			satisfy the additional condition
			\begin{align*}
				\bar\mu_i &\geq0 
				\qquad\forall i\in I_{\alpha=0}^{00}(\bar x),
				\\
				\bar\nu_i &\geq0 
				\qquad\forall i\in I_{\alpha=1}^{00}(\bar x).
			\end{align*}
		\item
			\label{item:def_p_stat}
			We call $\bar x$ 
			\emph{P-stationary with respect to $\alpha$} or
			\emph{\pstat\alpha-stationary}
			if it is weakly stationary and the multipliers $\bar\mu$, $\bar\nu$
			satisfy the additional condition
			\begin{align*}
				(\bar\mu_i \geq0 \land\bar\nu_i\geq0) \lor \bar\mu_i =0
% 				\qquad\text{if}\; i\in I^{00}(\bar x),\,\alpha_i=0,
				\qquad\forall i\in I_{\alpha=0}^{00}(\bar x),
				\\
				(\bar\mu_i \geq0 \land\bar\nu_i\geq0) \lor \bar\nu_i =0
% 				\qquad\text{if}\; i\in I^{00}(\bar x),\,\alpha_i = 1.
				\qquad\forall i\in I_{\alpha=1}^{00}(\bar x).
			\end{align*}
		\item
			\label{item:def_cdstat}
			We call $\bar x$ 
			\emph{P-stationary with respect to $d$} or
			\emph{\cstat{d}-stationary}
			if it is weakly stationary and the multipliers $\bar\mu$, $\bar\nu$
			satisfy the additional condition
			\begin{align*}
				(\bar\mu_i \geq0 \land\bar\nu_i\geq0) \lor 
				(1-d_i)\bar\mu_i = d_i\bar\nu_i
				\qquad\forall i\in I^{00}(\bar x).
			\end{align*}
	\end{enumerate}
\end{definition}
It is easy to see that the definitions of part~\ref{item:def_p_stat}
and part~\ref{item:def_cdstat} are consistent if $\alpha=d$.
Nonetheless, we decided to write down part~\ref{item:def_p_stat}, 
as it is an important special case
(because it is stronger than M-stationarity).
Throughout the article, we use the variable $\alpha$
for vectors in $\set{0,1}^p$, and $d$ for vectors in $[0,1]^p$.

The \astat\alpha-stationarity condition was already observed
in \cite[p.~610]{FlegelKanzow2005:3} when A-stationarity was introduced,
but we use the name \astat\alpha-stationarity for disambiguation.

Let us make the analogous definitions for \eqref{eq:mpvc}.
\begin{definition}
	\label{def:new_stat_mpvc}
	Let $\bar x\in\R^n$ be a feasible point of \eqref{eq:mpvc} and
	$\alpha\in\set{0,1}^p$, $d\in [0,1]^p$ be given.
	\begin{enumerate}
		\item
			\label{item:def:astat_mpvc}
			We call $\bar x$ \emph{\astat\alpha-stationary}
			if it is weakly stationary and the multipliers $\bar\mu$, $\bar\nu$
			satisfy
			\begin{align*}
				\bar\mu_i &=0 \qquad\forall i\in I_{\alpha=0}^{00}(\bar x),
				\\
				\bar\nu_i &\geq0 \qquad\forall i\in I_{\alpha=1}^{00}(\bar x).
			\end{align*}
		\item
			\label{item:def:p_stat_mpvc}
			We call $\bar x$ \emph{\pstat\alpha-stationary}
			if it is weakly stationary and the multipliers $\bar\mu$, $\bar\nu$
			satisfy
			\begin{align*}
				\bar\mu_i &=0 \qquad\forall i\in I_{\alpha=0}^{00}(\bar x),
				\\
% 				\bar\nu_i\geq0\land\bar\mu_i\bar\nu_i &=0
				(\bar\mu_i=0\land\bar\nu_i\geq0)\lor \bar\nu_i &=0
				\qquad\forall i\in I_{\alpha=1}^{00}(\bar x).
			\end{align*}
		\item
			\label{item:def:cd_stat_mpvc}
			We call $\bar x$ \cstat{d}-stationary
			if it is weakly stationary and the multipliers $\bar\mu$, $\bar\nu$
			satisfy
			\begin{align*}
				(\bar\mu_i =0 \land\bar\nu_i\geq0) \lor 
				(1-d_i)\bar\mu_i = -d_i\bar\nu_i
				\qquad\forall i\in I^{00}(\bar x).
			\end{align*}
	\end{enumerate}
\end{definition}
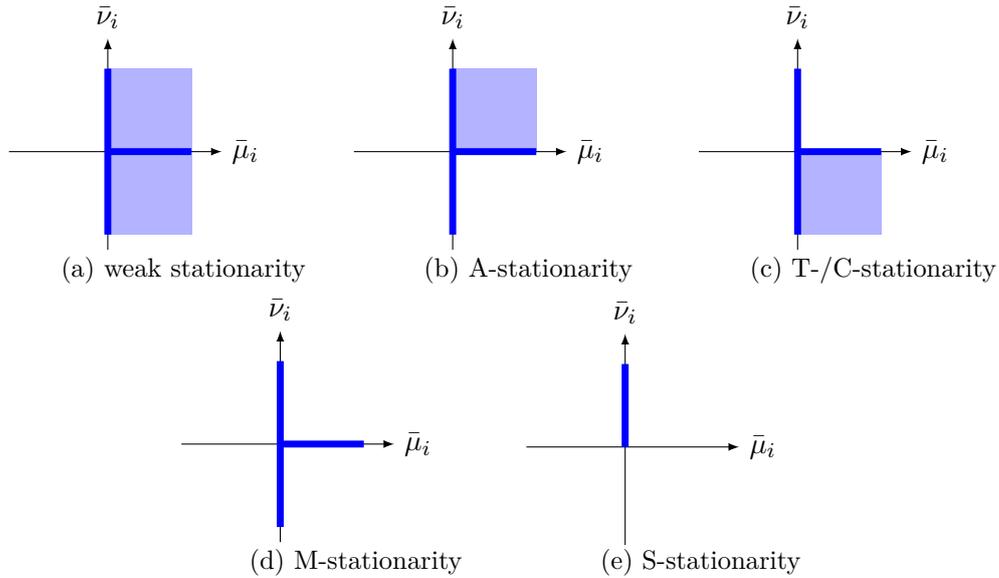
\begin{figure}[t]
	\centering
	\begin{subfigure}[b]{.3\textwidth}
		\def\castat{1}
		\def\aastat{1}
		\def\abstat{1}
		\centering
\begin{tikzpicture}[scale=1.1]
% 	\fill[blue, opacity=0.3]
% 	(1,0) -- (1,1) --(0,1) --(0,0) --cycle;
	\ifthenelse{\castat=1}{
		\fill[blue, opacity=0.3]
		(1,0) -- (1,-1) --(0,-1) --(0,0) --cycle;
	}{}
% 	\ifthenelse{\aastat=1}{
% 		\fill[blue, opacity=0.3]
% 		(1,0) -- (1,-1) --(0,-1) --(0,0) --cycle;
% 	}{}
	\ifthenelse{\abstat=1}{
		\fill[blue, opacity=0.3]
		(1,0) -- (1,1) --(0,1) --(0,0) --cycle;
	}{}
	\draw[axis] (-1,0) -- (1,0)
% 	node[right=2* \nudge cm] {$G(x)$\sometimes{, $\mu$}};
	node[right=2* \nudge cm] {$\bar\mu_i$};
	\draw[axis] (0,-1) -- (0,1)
% 	node[above=2*\nudge cm] {$H(x)$\sometimes{, $\nu$}};
	node[above=2*\nudge cm] {$\bar\nu_i$};
% 	\draw[line width=\feaslw,blue] (0,0) -- (1,0) ;
	\draw[line width=\feaslw,blue] (0,0) -- (0,1) ;
	\ifthenelse{\pastat=1 \OR \aastat=1 \OR \castat=1}{
		\draw[line width=\feaslw,blue] (0,-1) -- (0,0) ;
	}{}
	\ifthenelse{\cdstat=1}{
		\draw[line width=\feaslw,blue] (0,0) -- (1,-1) ;
	}{}
	\ifthenelse{\weirdstat=1}{
		\draw[line width=0.7*\feaslw,blue,domain=0:1,variable=\x ] 
		(0,0) -- plot ({(\x)}, {-\x*\x});
% 		(0,0) -- plot ({sqrt(\x)-\x/3}, {-\x});
	}{}
	\ifthenelse{\pbstat=1 \OR \abstat=1 \OR \castat=1}{
		\draw[line width=\feaslw,blue] (1,0) -- (0,0) ;
	}{}
\end{tikzpicture}
\def\pastat{0}
\def\pbstat{0}
\def\aastat{0}
\def\abstat{0}
\def\castat{0}
\def\cdstat{0}
\def\weirdstat{0}
		\caption{weak stationarity}
	\end{subfigure}
	\begin{subfigure}[b]{.3\textwidth}
		\def\aastat{1}
		\def\abstat{1}
		\centering
\begin{tikzpicture}[scale=1.1]
% 	\fill[blue, opacity=0.3]
% 	(1,0) -- (1,1) --(0,1) --(0,0) --cycle;
	\ifthenelse{\castat=1}{
		\fill[blue, opacity=0.3]
		(1,0) -- (1,-1) --(0,-1) --(0,0) --cycle;
	}{}
% 	\ifthenelse{\aastat=1}{
% 		\fill[blue, opacity=0.3]
% 		(1,0) -- (1,-1) --(0,-1) --(0,0) --cycle;
% 	}{}
	\ifthenelse{\abstat=1}{
		\fill[blue, opacity=0.3]
		(1,0) -- (1,1) --(0,1) --(0,0) --cycle;
	}{}
	\draw[axis] (-1,0) -- (1,0)
% 	node[right=2* \nudge cm] {$G(x)$\sometimes{, $\mu$}};
	node[right=2* \nudge cm] {$\bar\mu_i$};
	\draw[axis] (0,-1) -- (0,1)
% 	node[above=2*\nudge cm] {$H(x)$\sometimes{, $\nu$}};
	node[above=2*\nudge cm] {$\bar\nu_i$};
% 	\draw[line width=\feaslw,blue] (0,0) -- (1,0) ;
	\draw[line width=\feaslw,blue] (0,0) -- (0,1) ;
	\ifthenelse{\pastat=1 \OR \aastat=1 \OR \castat=1}{
		\draw[line width=\feaslw,blue] (0,-1) -- (0,0) ;
	}{}
	\ifthenelse{\cdstat=1}{
		\draw[line width=\feaslw,blue] (0,0) -- (1,-1) ;
	}{}
	\ifthenelse{\weirdstat=1}{
		\draw[line width=0.7*\feaslw,blue,domain=0:1,variable=\x ] 
		(0,0) -- plot ({(\x)}, {-\x*\x});
% 		(0,0) -- plot ({sqrt(\x)-\x/3}, {-\x});
	}{}
	\ifthenelse{\pbstat=1 \OR \abstat=1 \OR \castat=1}{
		\draw[line width=\feaslw,blue] (1,0) -- (0,0) ;
	}{}
\end{tikzpicture}
\def\pastat{0}
\def\pbstat{0}
\def\aastat{0}
\def\abstat{0}
\def\castat{0}
\def\cdstat{0}
\def\weirdstat{0}
		\caption{A-stationarity}
	\end{subfigure}
	\begin{subfigure}[b]{.3\textwidth}
		\def\castat{1}
		\centering
\begin{tikzpicture}[scale=1.1]
% 	\fill[blue, opacity=0.3]
% 	(1,0) -- (1,1) --(0,1) --(0,0) --cycle;
	\ifthenelse{\castat=1}{
		\fill[blue, opacity=0.3]
		(1,0) -- (1,-1) --(0,-1) --(0,0) --cycle;
	}{}
% 	\ifthenelse{\aastat=1}{
% 		\fill[blue, opacity=0.3]
% 		(1,0) -- (1,-1) --(0,-1) --(0,0) --cycle;
% 	}{}
	\ifthenelse{\abstat=1}{
		\fill[blue, opacity=0.3]
		(1,0) -- (1,1) --(0,1) --(0,0) --cycle;
	}{}
	\draw[axis] (-1,0) -- (1,0)
% 	node[right=2* \nudge cm] {$G(x)$\sometimes{, $\mu$}};
	node[right=2* \nudge cm] {$\bar\mu_i$};
	\draw[axis] (0,-1) -- (0,1)
% 	node[above=2*\nudge cm] {$H(x)$\sometimes{, $\nu$}};
	node[above=2*\nudge cm] {$\bar\nu_i$};
% 	\draw[line width=\feaslw,blue] (0,0) -- (1,0) ;
	\draw[line width=\feaslw,blue] (0,0) -- (0,1) ;
	\ifthenelse{\pastat=1 \OR \aastat=1 \OR \castat=1}{
		\draw[line width=\feaslw,blue] (0,-1) -- (0,0) ;
	}{}
	\ifthenelse{\cdstat=1}{
		\draw[line width=\feaslw,blue] (0,0) -- (1,-1) ;
	}{}
	\ifthenelse{\weirdstat=1}{
		\draw[line width=0.7*\feaslw,blue,domain=0:1,variable=\x ] 
		(0,0) -- plot ({(\x)}, {-\x*\x});
% 		(0,0) -- plot ({sqrt(\x)-\x/3}, {-\x});
	}{}
	\ifthenelse{\pbstat=1 \OR \abstat=1 \OR \castat=1}{
		\draw[line width=\feaslw,blue] (1,0) -- (0,0) ;
	}{}
\end{tikzpicture}
\def\pastat{0}
\def\pbstat{0}
\def\aastat{0}
\def\abstat{0}
\def\castat{0}
\def\cdstat{0}
\def\weirdstat{0}
		\caption{T-/C-stationarity}
	\end{subfigure}
	\begin{subfigure}[b]{.3\textwidth}
		\def\pastat{1}
		\def\pbstat{1}
		\centering
\begin{tikzpicture}[scale=1.1]
% 	\fill[blue, opacity=0.3]
% 	(1,0) -- (1,1) --(0,1) --(0,0) --cycle;
	\ifthenelse{\castat=1}{
		\fill[blue, opacity=0.3]
		(1,0) -- (1,-1) --(0,-1) --(0,0) --cycle;
	}{}
% 	\ifthenelse{\aastat=1}{
% 		\fill[blue, opacity=0.3]
% 		(1,0) -- (1,-1) --(0,-1) --(0,0) --cycle;
% 	}{}
	\ifthenelse{\abstat=1}{
		\fill[blue, opacity=0.3]
		(1,0) -- (1,1) --(0,1) --(0,0) --cycle;
	}{}
	\draw[axis] (-1,0) -- (1,0)
% 	node[right=2* \nudge cm] {$G(x)$\sometimes{, $\mu$}};
	node[right=2* \nudge cm] {$\bar\mu_i$};
	\draw[axis] (0,-1) -- (0,1)
% 	node[above=2*\nudge cm] {$H(x)$\sometimes{, $\nu$}};
	node[above=2*\nudge cm] {$\bar\nu_i$};
% 	\draw[line width=\feaslw,blue] (0,0) -- (1,0) ;
	\draw[line width=\feaslw,blue] (0,0) -- (0,1) ;
	\ifthenelse{\pastat=1 \OR \aastat=1 \OR \castat=1}{
		\draw[line width=\feaslw,blue] (0,-1) -- (0,0) ;
	}{}
	\ifthenelse{\cdstat=1}{
		\draw[line width=\feaslw,blue] (0,0) -- (1,-1) ;
	}{}
	\ifthenelse{\weirdstat=1}{
		\draw[line width=0.7*\feaslw,blue,domain=0:1,variable=\x ] 
		(0,0) -- plot ({(\x)}, {-\x*\x});
% 		(0,0) -- plot ({sqrt(\x)-\x/3}, {-\x});
	}{}
	\ifthenelse{\pbstat=1 \OR \abstat=1 \OR \castat=1}{
		\draw[line width=\feaslw,blue] (1,0) -- (0,0) ;
	}{}
\end{tikzpicture}
\def\pastat{0}
\def\pbstat{0}
\def\aastat{0}
\def\abstat{0}
\def\castat{0}
\def\cdstat{0}
\def\weirdstat{0}
		\caption{M-stationarity}
		\label{fig:mstat_mpvc}
	\end{subfigure}
	\begin{subfigure}[b]{.3\textwidth}
		\centering
\begin{tikzpicture}[scale=1.1]
% 	\fill[blue, opacity=0.3]
% 	(1,0) -- (1,1) --(0,1) --(0,0) --cycle;
	\ifthenelse{\castat=1}{
		\fill[blue, opacity=0.3]
		(1,0) -- (1,-1) --(0,-1) --(0,0) --cycle;
	}{}
% 	\ifthenelse{\aastat=1}{
% 		\fill[blue, opacity=0.3]
% 		(1,0) -- (1,-1) --(0,-1) --(0,0) --cycle;
% 	}{}
	\ifthenelse{\abstat=1}{
		\fill[blue, opacity=0.3]
		(1,0) -- (1,1) --(0,1) --(0,0) --cycle;
	}{}
	\draw[axis] (-1,0) -- (1,0)
% 	node[right=2* \nudge cm] {$G(x)$\sometimes{, $\mu$}};
	node[right=2* \nudge cm] {$\bar\mu_i$};
	\draw[axis] (0,-1) -- (0,1)
% 	node[above=2*\nudge cm] {$H(x)$\sometimes{, $\nu$}};
	node[above=2*\nudge cm] {$\bar\nu_i$};
% 	\draw[line width=\feaslw,blue] (0,0) -- (1,0) ;
	\draw[line width=\feaslw,blue] (0,0) -- (0,1) ;
	\ifthenelse{\pastat=1 \OR \aastat=1 \OR \castat=1}{
		\draw[line width=\feaslw,blue] (0,-1) -- (0,0) ;
	}{}
	\ifthenelse{\cdstat=1}{
		\draw[line width=\feaslw,blue] (0,0) -- (1,-1) ;
	}{}
	\ifthenelse{\weirdstat=1}{
		\draw[line width=0.7*\feaslw,blue,domain=0:1,variable=\x ] 
		(0,0) -- plot ({(\x)}, {-\x*\x});
% 		(0,0) -- plot ({sqrt(\x)-\x/3}, {-\x});
	}{}
	\ifthenelse{\pbstat=1 \OR \abstat=1 \OR \castat=1}{
		\draw[line width=\feaslw,blue] (1,0) -- (0,0) ;
	}{}
\end{tikzpicture}
\def\pastat{0}
\def\pbstat{0}
\def\aastat{0}
\def\abstat{0}
\def\castat{0}
\def\cdstat{0}
\def\weirdstat{0}
		\caption{S-stationarity}
	\end{subfigure}
	\caption{geometric illustration of stationarity conditions
	for MPVCs with $i\in I^{00}(\bar x)$}
	\label{fig:stat_mpvc}
\end{figure}
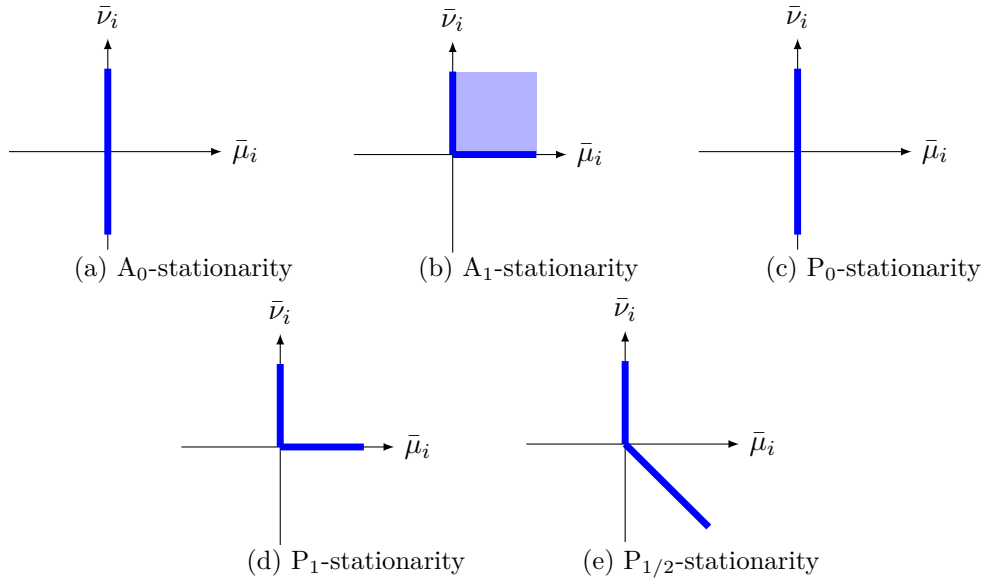
\begin{figure}[t]
	\centering
	\begin{subfigure}[b]{.3\textwidth}
		\def\aastat{1}
		\centering
\begin{tikzpicture}[scale=1.1]
% 	\fill[blue, opacity=0.3]
% 	(1,0) -- (1,1) --(0,1) --(0,0) --cycle;
	\ifthenelse{\castat=1}{
		\fill[blue, opacity=0.3]
		(1,0) -- (1,-1) --(0,-1) --(0,0) --cycle;
	}{}
% 	\ifthenelse{\aastat=1}{
% 		\fill[blue, opacity=0.3]
% 		(1,0) -- (1,-1) --(0,-1) --(0,0) --cycle;
% 	}{}
	\ifthenelse{\abstat=1}{
		\fill[blue, opacity=0.3]
		(1,0) -- (1,1) --(0,1) --(0,0) --cycle;
	}{}
	\draw[axis] (-1,0) -- (1,0)
% 	node[right=2* \nudge cm] {$G(x)$\sometimes{, $\mu$}};
	node[right=2* \nudge cm] {$\bar\mu_i$};
	\draw[axis] (0,-1) -- (0,1)
% 	node[above=2*\nudge cm] {$H(x)$\sometimes{, $\nu$}};
	node[above=2*\nudge cm] {$\bar\nu_i$};
% 	\draw[line width=\feaslw,blue] (0,0) -- (1,0) ;
	\draw[line width=\feaslw,blue] (0,0) -- (0,1) ;
	\ifthenelse{\pastat=1 \OR \aastat=1 \OR \castat=1}{
		\draw[line width=\feaslw,blue] (0,-1) -- (0,0) ;
	}{}
	\ifthenelse{\cdstat=1}{
		\draw[line width=\feaslw,blue] (0,0) -- (1,-1) ;
	}{}
	\ifthenelse{\weirdstat=1}{
		\draw[line width=0.7*\feaslw,blue,domain=0:1,variable=\x ] 
		(0,0) -- plot ({(\x)}, {-\x*\x});
% 		(0,0) -- plot ({sqrt(\x)-\x/3}, {-\x});
	}{}
	\ifthenelse{\pbstat=1 \OR \abstat=1 \OR \castat=1}{
		\draw[line width=\feaslw,blue] (1,0) -- (0,0) ;
	}{}
\end{tikzpicture}
\def\pastat{0}
\def\pbstat{0}
\def\aastat{0}
\def\abstat{0}
\def\castat{0}
\def\cdstat{0}
\def\weirdstat{0}
		\caption{\astat0-stationarity}
		\label{fig:astat0_mpvc}
	\end{subfigure}
	\begin{subfigure}[b]{.3\textwidth}
		\def\abstat{1}
		\centering
\begin{tikzpicture}[scale=1.1]
% 	\fill[blue, opacity=0.3]
% 	(1,0) -- (1,1) --(0,1) --(0,0) --cycle;
	\ifthenelse{\castat=1}{
		\fill[blue, opacity=0.3]
		(1,0) -- (1,-1) --(0,-1) --(0,0) --cycle;
	}{}
% 	\ifthenelse{\aastat=1}{
% 		\fill[blue, opacity=0.3]
% 		(1,0) -- (1,-1) --(0,-1) --(0,0) --cycle;
% 	}{}
	\ifthenelse{\abstat=1}{
		\fill[blue, opacity=0.3]
		(1,0) -- (1,1) --(0,1) --(0,0) --cycle;
	}{}
	\draw[axis] (-1,0) -- (1,0)
% 	node[right=2* \nudge cm] {$G(x)$\sometimes{, $\mu$}};
	node[right=2* \nudge cm] {$\bar\mu_i$};
	\draw[axis] (0,-1) -- (0,1)
% 	node[above=2*\nudge cm] {$H(x)$\sometimes{, $\nu$}};
	node[above=2*\nudge cm] {$\bar\nu_i$};
% 	\draw[line width=\feaslw,blue] (0,0) -- (1,0) ;
	\draw[line width=\feaslw,blue] (0,0) -- (0,1) ;
	\ifthenelse{\pastat=1 \OR \aastat=1 \OR \castat=1}{
		\draw[line width=\feaslw,blue] (0,-1) -- (0,0) ;
	}{}
	\ifthenelse{\cdstat=1}{
		\draw[line width=\feaslw,blue] (0,0) -- (1,-1) ;
	}{}
	\ifthenelse{\weirdstat=1}{
		\draw[line width=0.7*\feaslw,blue,domain=0:1,variable=\x ] 
		(0,0) -- plot ({(\x)}, {-\x*\x});
% 		(0,0) -- plot ({sqrt(\x)-\x/3}, {-\x});
	}{}
	\ifthenelse{\pbstat=1 \OR \abstat=1 \OR \castat=1}{
		\draw[line width=\feaslw,blue] (1,0) -- (0,0) ;
	}{}
\end{tikzpicture}
\def\pastat{0}
\def\pbstat{0}
\def\aastat{0}
\def\abstat{0}
\def\castat{0}
\def\cdstat{0}
\def\weirdstat{0}
		\caption{\astat1-stationarity}
	\end{subfigure}
	\begin{subfigure}[b]{.3\textwidth}
		\def\pastat{1}
		\centering
\begin{tikzpicture}[scale=1.1]
% 	\fill[blue, opacity=0.3]
% 	(1,0) -- (1,1) --(0,1) --(0,0) --cycle;
	\ifthenelse{\castat=1}{
		\fill[blue, opacity=0.3]
		(1,0) -- (1,-1) --(0,-1) --(0,0) --cycle;
	}{}
% 	\ifthenelse{\aastat=1}{
% 		\fill[blue, opacity=0.3]
% 		(1,0) -- (1,-1) --(0,-1) --(0,0) --cycle;
% 	}{}
	\ifthenelse{\abstat=1}{
		\fill[blue, opacity=0.3]
		(1,0) -- (1,1) --(0,1) --(0,0) --cycle;
	}{}
	\draw[axis] (-1,0) -- (1,0)
% 	node[right=2* \nudge cm] {$G(x)$\sometimes{, $\mu$}};
	node[right=2* \nudge cm] {$\bar\mu_i$};
	\draw[axis] (0,-1) -- (0,1)
% 	node[above=2*\nudge cm] {$H(x)$\sometimes{, $\nu$}};
	node[above=2*\nudge cm] {$\bar\nu_i$};
% 	\draw[line width=\feaslw,blue] (0,0) -- (1,0) ;
	\draw[line width=\feaslw,blue] (0,0) -- (0,1) ;
	\ifthenelse{\pastat=1 \OR \aastat=1 \OR \castat=1}{
		\draw[line width=\feaslw,blue] (0,-1) -- (0,0) ;
	}{}
	\ifthenelse{\cdstat=1}{
		\draw[line width=\feaslw,blue] (0,0) -- (1,-1) ;
	}{}
	\ifthenelse{\weirdstat=1}{
		\draw[line width=0.7*\feaslw,blue,domain=0:1,variable=\x ] 
		(0,0) -- plot ({(\x)}, {-\x*\x});
% 		(0,0) -- plot ({sqrt(\x)-\x/3}, {-\x});
	}{}
	\ifthenelse{\pbstat=1 \OR \abstat=1 \OR \castat=1}{
		\draw[line width=\feaslw,blue] (1,0) -- (0,0) ;
	}{}
\end{tikzpicture}
\def\pastat{0}
\def\pbstat{0}
\def\aastat{0}
\def\abstat{0}
\def\castat{0}
\def\cdstat{0}
\def\weirdstat{0}
		\caption{\pstat0-stationarity}
	\end{subfigure}
	\begin{subfigure}[b]{.3\textwidth}
		\def\pbstat{1}
		\centering
\begin{tikzpicture}[scale=1.1]
% 	\fill[blue, opacity=0.3]
% 	(1,0) -- (1,1) --(0,1) --(0,0) --cycle;
	\ifthenelse{\castat=1}{
		\fill[blue, opacity=0.3]
		(1,0) -- (1,-1) --(0,-1) --(0,0) --cycle;
	}{}
% 	\ifthenelse{\aastat=1}{
% 		\fill[blue, opacity=0.3]
% 		(1,0) -- (1,-1) --(0,-1) --(0,0) --cycle;
% 	}{}
	\ifthenelse{\abstat=1}{
		\fill[blue, opacity=0.3]
		(1,0) -- (1,1) --(0,1) --(0,0) --cycle;
	}{}
	\draw[axis] (-1,0) -- (1,0)
% 	node[right=2* \nudge cm] {$G(x)$\sometimes{, $\mu$}};
	node[right=2* \nudge cm] {$\bar\mu_i$};
	\draw[axis] (0,-1) -- (0,1)
% 	node[above=2*\nudge cm] {$H(x)$\sometimes{, $\nu$}};
	node[above=2*\nudge cm] {$\bar\nu_i$};
% 	\draw[line width=\feaslw,blue] (0,0) -- (1,0) ;
	\draw[line width=\feaslw,blue] (0,0) -- (0,1) ;
	\ifthenelse{\pastat=1 \OR \aastat=1 \OR \castat=1}{
		\draw[line width=\feaslw,blue] (0,-1) -- (0,0) ;
	}{}
	\ifthenelse{\cdstat=1}{
		\draw[line width=\feaslw,blue] (0,0) -- (1,-1) ;
	}{}
	\ifthenelse{\weirdstat=1}{
		\draw[line width=0.7*\feaslw,blue,domain=0:1,variable=\x ] 
		(0,0) -- plot ({(\x)}, {-\x*\x});
% 		(0,0) -- plot ({sqrt(\x)-\x/3}, {-\x});
	}{}
	\ifthenelse{\pbstat=1 \OR \abstat=1 \OR \castat=1}{
		\draw[line width=\feaslw,blue] (1,0) -- (0,0) ;
	}{}
\end{tikzpicture}
\def\pastat{0}
\def\pbstat{0}
\def\aastat{0}
\def\abstat{0}
\def\castat{0}
\def\cdstat{0}
\def\weirdstat{0}
		\caption{\pstat1-stationarity}
	\end{subfigure}
	\begin{subfigure}[b]{.3\textwidth}
		\def\cdstat{1}
		\centering
\begin{tikzpicture}[scale=1.1]
% 	\fill[blue, opacity=0.3]
% 	(1,0) -- (1,1) --(0,1) --(0,0) --cycle;
	\ifthenelse{\castat=1}{
		\fill[blue, opacity=0.3]
		(1,0) -- (1,-1) --(0,-1) --(0,0) --cycle;
	}{}
% 	\ifthenelse{\aastat=1}{
% 		\fill[blue, opacity=0.3]
% 		(1,0) -- (1,-1) --(0,-1) --(0,0) --cycle;
% 	}{}
	\ifthenelse{\abstat=1}{
		\fill[blue, opacity=0.3]
		(1,0) -- (1,1) --(0,1) --(0,0) --cycle;
	}{}
	\draw[axis] (-1,0) -- (1,0)
% 	node[right=2* \nudge cm] {$G(x)$\sometimes{, $\mu$}};
	node[right=2* \nudge cm] {$\bar\mu_i$};
	\draw[axis] (0,-1) -- (0,1)
% 	node[above=2*\nudge cm] {$H(x)$\sometimes{, $\nu$}};
	node[above=2*\nudge cm] {$\bar\nu_i$};
% 	\draw[line width=\feaslw,blue] (0,0) -- (1,0) ;
	\draw[line width=\feaslw,blue] (0,0) -- (0,1) ;
	\ifthenelse{\pastat=1 \OR \aastat=1 \OR \castat=1}{
		\draw[line width=\feaslw,blue] (0,-1) -- (0,0) ;
	}{}
	\ifthenelse{\cdstat=1}{
		\draw[line width=\feaslw,blue] (0,0) -- (1,-1) ;
	}{}
	\ifthenelse{\weirdstat=1}{
		\draw[line width=0.7*\feaslw,blue,domain=0:1,variable=\x ] 
		(0,0) -- plot ({(\x)}, {-\x*\x});
% 		(0,0) -- plot ({sqrt(\x)-\x/3}, {-\x});
	}{}
	\ifthenelse{\pbstat=1 \OR \abstat=1 \OR \castat=1}{
		\draw[line width=\feaslw,blue] (1,0) -- (0,0) ;
	}{}
\end{tikzpicture}
\def\pastat{0}
\def\pbstat{0}
\def\aastat{0}
\def\abstat{0}
\def\castat{0}
\def\cdstat{0}
\def\weirdstat{0}
		\caption{\cstat{1/2}-stationarity}
	\end{subfigure}
	\caption{geometric illustration of stationarity conditions
	for MPVCs with $i\in I^{00}(\bar x)$}
	\label{fig:new_stat_mpvc}
\end{figure}
Recall that $\bar\mu_i\geq0$ holds for $i\in I^{00}(\bar x)$
for weakly stationary points of \eqref{eq:mpvc}.
Again, one can see that parts~\ref{item:def:p_stat_mpvc}
and parts~\ref{item:def:cd_stat_mpvc}
are consistent if $d=\alpha$.

With the exception of linearized B-stationarity,
the stationarity conditions of
\cref{def:stat,def:mpvc_stat,def:new_stat_mpcc,def:new_stat_mpvc}
are illustrated in
\cref{fig:new_stat,fig:wac_stat_mpcc,fig:stat_mpvc,fig:new_stat_mpvc}.

The following relations for the new
stationarity conditions follow directly from
the definitions.
\begin{corollary}
	\label{cor:relations}
	Let $\bar x$ be a feasible point of 
	\eqref{eq:mpcc} or \eqref{eq:mpvc}.
	\begin{enumerate}
% 		\item
% 			\label{item:cd_implies_palpha}
% 			If $\bar x$ is \cstat{d}-stationary for all $d\in [0,1]^p$,
% 			then it is \pstat\alpha-stationary for all $\alpha\in\set{0,1}^p$.
		\item
			\label{item:pstat_implies_astat}
			If $\bar x$ is \pstat\alpha-stationary for some
			$\alpha\in\set{0,1}^p$, then it is \astat\alpha-stationary.
		\item
			The point $\bar x$ is A-stationary if and only if
			there exists some $\alpha\in\set{0,1}^p$ such that
			$\bar x$ is \astat\alpha-stationary.
		\item
			The point $\bar x$ is C-stationary if and only if
			there exists some $d\in [0,1]^p$ such that
			$\bar x$ is \cstat{d}-stationary.
		\item
			The point $\bar x$ is M-stationary if and only if
			there exists some $\alpha\in\set{0,1}^p$ such that
			$\bar x$ is \pstat\alpha-stationary.
		\item
			In the case of MPVCs, \astat0-stationarity is the same as
			\pstat0-stationarity.
	\end{enumerate}
\end{corollary}

\subsection{Auxiliary optimization problems}

If $\bar x$ is a feasible point of \eqref{eq:mpcc} and
$\alpha\in\set{0,1}^p$,
then we introduce the auxiliary nonlinear optimization problem
\begin{equation*}
	\labelparamtag{eq:nlp_mpcc}{\text{NLP}}{\alpha}
	\begin{minproblem}[x\in\R^n]{f(x)}
		g(x)&\leq 0,
		& h(x)&=0,
		\\
		G_i(x) &\geq0, & H_i(x) &= 0
		\qquad\forall i\in 
		I^{+0}(\bar x)\cup I_{\alpha=0}^{00}(\bar x),
		\\
		G_i(x) &=0, & H_i(x) &\geq 0
		\qquad\forall i\in 
		I^{0+}(\bar x)\cup I_{\alpha=1}^{00}(\bar x).
	\end{minproblem}
\end{equation*}
Note that the feasible set of this auxiliary problem
is a subset of the feasible set of \eqref{eq:mpcc}.
This optimization problem can also be found in
\cite[Section~2]{PangFukushima1999}.
% Let us also write down the KKT conditions at $\bar x$
% for this problem, with multipliers
% $\bar\lambda\in\R^l$, $\bar\eta\in\R^m$,
% $\bar\mu,\bar\nu\in\R^p$.
% \begin{equation*}
% 	\begin{aligned}
% 		&&&&\mathllap{
% 			\nabla f(\bar x) 
% 			+ \sum_{i\in I^l} \bar\lambda_i\nabla g_i(\bar x)
% 			+ \sum_{i\in I^m} \bar\eta_i\nabla h_i(\bar x)
% 			- \sum_{i\in I^p} \bar\mu_i\nabla G_i(\bar x)
% 			- \sum_{i\in I^p} \bar\nu_i\nabla H_i(\bar x)
% 		}
% 		&=0, \\
% 		&\hspace{18em}&&\forall i\in I^g(\bar x):
% 		&
% 		\bar\lambda_i &\geq0,
% 		\\
% 		&&&\forall i\in I^l\setminus I^g(\bar x):
% 		&
% 		\bar\lambda_i &=0,
% 		\\
% 		&&&\forall i\in I^{+0}(\bar x):
% 		& \bar\mu_i  &=0,
% 		\\
% 		&&&\forall i\in I_{\alpha=0}^{00}(\bar x):
% 		& \bar\mu_i  &\geq0,
% 		\\
% 		&&&\forall i\in I_{\alpha=1}^{00}(\bar x):
% 		& \bar\nu_i  &\geq0,
% 		\\
% 		&&&\forall i\in I^{0+}(\bar x):
% 		& \bar\nu_i  &=0,
% 	\end{aligned}
% \end{equation*}
% 

Similarly,
if $\bar x$ is a feasible point of \eqref{eq:mpvc},
then we introduce the auxiliary nonlinear optimization problem
\begin{equation*}
	\labelparamtag{eq:nlp_mpvc}{\text{NLP}}{\alpha}
	\begin{minproblem}[x\in\R^n]{f(x)}
		g(x)&\leq 0, & h(x)&=0,
		\\
		&& H_i(x) &= 0
		\qquad\forall i\in 
		I^{+0}(\bar x)\cup I_{\alpha=0}^{00}(\bar x),
		\\
		G_i(x) &\leq 0, &H_i(x) &\geq0
		\qquad\forall i\in 
		I^{-0}(\bar x)\cup I^{-+}(\bar x)\cup I^{0+}(\bar x)
		\cup I_{\alpha=1}^{00}(\bar x).
	\end{minproblem}
\end{equation*}
This optimization problem
can also be found in \cite[(7)]{HoheiselKanzow2008}.

Note that these auxiliary problems depend on $\bar x$
(in both the MPCC and the MPVC case).
It will be clear from context, whether the MPCC-version
of \paramref{eq:nlp_mpcc}{\alpha} or the MPVC-version
of \paramref{eq:nlp_mpvc}{\alpha} is meant.

In both cases we denote the tangent cone 
at $\bar x$ to the feasible set of
\paramref{eq:nlp_mpcc}{\alpha}
by $\tnlp\alpha(\bar x)$,
and the standard linearization cone by $\tnlplin\alpha(\bar x)$.
Note that $\tnlplin\alpha(\bar x)$ is a convex cone,
whereas $\tmpcclin(\bar x)$ and $\tmpvclin(\bar x)$
are not always convex.
Recall that the \emph{Guignard constraint qualification}
or \emph{GCQ} is satisfied at $\bar x$ for
\paramref{eq:nlp_mpcc}{\alpha}
if $\tnlp\alpha(\bar x)\polar=\tnlplin\alpha(\bar x)\polar$.

Finally, we state a simple lemma with well-known facts
from the basic theory of nonlinear programming.
\begin{lemma}
	\label{lem:basic_nlp_theory}
	\begin{enumerate}
		\item
			\label{item:b_stat}
			If $\bar x$ is a local minimizer of \eqref{eq:mpcc},
			then $-\nabla f(\bar x)\in\tmpcc(\bar x)\polar$ holds.
			If $\bar x$ is a local minimizer of \eqref{eq:mpvc},
			then $-\nabla f(\bar x)\in\tmpvc(\bar x)\polar$ holds.
		\item
			\label{item:lin_implies_kkt}
			For some $\alpha\in\set{0,1}^p$ and feasible $\bar x$,
			$-\nabla f(\bar x)\in\tnlplin{\alpha}(\bar x)\polar$
			holds if and only if
			there exist multipliers
			$(\eta^\alpha,\lambda^\alpha,\mu^\alpha,\nu^\alpha)$
			that satisfy the KKT conditions
			of \paramref{eq:nlp_mpcc}{\alpha},
			which is the same as the system for \astat\alpha-stationarity
			(this holds for both of the MPCC and MPVC version of 
			\paramref{eq:nlp_mpcc}{\alpha}).
	\end{enumerate}
\end{lemma}
\begin{proof}
	Part~\ref{item:b_stat} can be shown using the definition
	of the tangent cone and polar cone,
	and part~\ref{item:lin_implies_kkt}
	can be shown by calculating $\tnlplin{\alpha}(\bar x)\polar$,
	e.g.\ using Farkas' Lemma.
	That the KKT system
	of \paramref{eq:nlp_mpcc}{\alpha}
	is the same as the system for \astat\alpha-stationarity
	can be seen by writing down both systems.
	For more detailed proofs, we refer to the standard literature.
\end{proof}

\section{New stationarity systems between M- and S-stationarity}
\label{sec:between_mpcc}

Let us start by stating that 
linearized B-stationarity and
\astat\alpha-stationarity
are indeed a stationarity condition under MPCC-GCQ.
The result can (partially) be obtained from the proof of
\cite[Theorem~3.4]{FlegelKanzow2005:3},
and \astat\alpha-stationarity was also shown in
\cite[Proposition~3.1]{Harder2020}.
The equivalence was also mentioned in
\cite[Section~2.1]{ScheelScholtes2000}.
We include a proof for convenience.
\begin{proposition}
	\label{prop:astat}
	Let $\bar x\in\R^n$ be a feasible point of \eqref{eq:mpcc}.
	\begin{enumerate}
		\item
			\label{item:bstat_mpcc}
			If $\bar x$ is a local minimizer of \eqref{eq:mpcc}
			that satisfies MPCC-GCQ, then $\bar x$ is
			linearized B-stationary.
		\item
			\label{item:bstat_iff_astat}
			The point $\bar x$ is linearized B-stationary if and only if it is
			\astat\alpha-stationary for all $\alpha\in\set{0,1}^p$.
	\end{enumerate}
	In particular, if $\bar x$ is a local minimizer of \eqref{eq:mpcc}
	that satisfies MPCC-GCQ,
	then it is \astat\alpha-stationary
	for all $\alpha\in\set{0,1}^p$.
\end{proposition}
\begin{proof}
% 	From the definition of $\tmpcc(\bar x)$
% 	and the fact that $\bar x$ is a local minimizer of
% 	\eqref{eq:mpcc} it can be concluded that the condition
% 	\begin{equation}
% 		\label{eq:bstat}
% 		\nabla f(\bar x)^\top d\geq0
% 		\qquad\forall\, d\in\tmpcc(\bar x)
% 	\end{equation}
% 	is satisfied.
% 	Thus, we have
% 	$-\nabla f(\bar x)\in \tmpcc(\bar x)\polar
% 	=\tmpcclin(\bar x)\polar$, 
% 	which shows linearized B-stationarity of $\bar x$.
	For part~\ref{item:bstat_mpcc}, we obtain
	$-\nabla f(\bar x)\in \tmpcc(\bar x)\polar
	=\tmpcclin(\bar x)\polar$
	from \cref{lem:basic_nlp_theory}~\ref{item:b_stat} and MPCC-GCQ.
	Thus, $\bar x$ is linearized B-stationary.

	For part~\ref{item:bstat_iff_astat},
	we first observe the equality
	\begin{equation*}
		\tmpcclin(\bar x)
		=\bigcup_{\alpha\in\set{0,1}^p} \tnlplin{\alpha}(\bar x),
	\end{equation*}
	which can be shown by direct calculations
	or obtained from \cite[Lemma~3.1]{FlegelKanzow2005:3}.
	Therefore, linearized B-stationarity can be written as
	\begin{equation*}
		-\nabla f(\bar x)\in \tmpcclin(\bar x)\polar
		= \paren[\bigg]{ 
		\bigcup_{\alpha\in\set{0,1}^p} \tnlplin{\alpha}(\bar x)}\polar
		= \bigcap_{\alpha\in\set{0,1}^p}\tnlplin{\alpha}(\bar x)\polar.
	\end{equation*}
	Thus, $\bar x$ is linearized B-stationary if and only if
	$-\nabla f(\bar x)\in \tnlplin{\alpha}(\bar x)\polar$
	for all $\alpha\in\set{0,1}^p$.
	However, the latter condition is equivalent to
	\astat\alpha-stationarity of $\bar x$
	by \cref{lem:basic_nlp_theory}~\ref{item:lin_implies_kkt}.
\end{proof}
In \cite{Harder2020}, the idea was to consider
convex combinations of the multipliers corresponding to the
\astat\alpha-stationarity system.
Here, we will use the same idea, but aim for stronger results.
An important ingredient is the Poincaré--Miranda theorem
which is a generalization of the intermediate value theorem.
\begin{theorem}[Poincaré--Miranda Theorem]
	\label{thm:poincare_miranda}
	Let $h:[0,1]^p\to\R^p$ be a continuous function such that
	\begin{align*}
		h_i(y) &\leq0 \qquad\text{if}\; y_i=0,
		\\
		h_i(y) &\geq0 \qquad\text{if}\; y_i=1
	\end{align*}
	holds for all $i\in I^p$.
	Then there exists a point $\bar y\in [0,1]^p$ with $h(\bar y)=0$.
\end{theorem}
We refer to \cite{Kulpa1997} for a proof of this theorem.
We mention that this theorem is equivalent to the
well-known Brouwer fixed-point theorem, see \cite{Miranda1940}.
\begin{lemma}
	\label{lem:a_implies_psi}
% 	Let $\set{\psi_i}_{i=1}^p$ be a family of continuous functions
% 	with $\psi_i:\R^2\to\R$ such that
% 	Let $\hat I\in\set{0,1}^p$ be an index set and
	Let $\bar x\in\R^n$ be a point and
	let $\psi:\R^p\times\R^p\to\R^p$ be a continuous function
	with the property
	\begin{equation}
		\label{eq:psi_sign_condition}
		\begin{aligned}
			\psi_i(a,b)&\leq0
			\quad\text{if}\;a_i\geq0
			\\
			\psi_i(a,b)&\geq0
			\quad\text{if}\;b_i\geq0
		\end{aligned}
		\qquad\forall\, a,b\in \R^p,\, i\in I^{00}(\bar x).
	\end{equation}
	Furthermore, for all $\alpha\in\set{0,1}^p$,
	let points $(\mu^\alpha,\nu^\alpha)\in A^\alpha$ be given, 
	where the set $A^\alpha$ is described via
	\begin{equation}
		\label{eq:def_set_A_alpha}
		A^\alpha :=
		\set{(\mu,\nu)\in\R^{2p}\given 
			\mu_i\geq0\;\forall i\in I_{\alpha=0}^{00}(\bar x),\;
			\nu_i\geq0\;\forall i\in I_{\alpha=1}^{00}(\bar x)
		}.
	\end{equation}
	Then there exists a point $(\bar\mu,\bar\nu)$
	in the convex hull of 
	$\set{(\mu^\alpha,\nu^\alpha)\given \alpha\in\set{0,1}^p}$
	with
	\begin{equation*}
		\psi_i(\bar\mu,\bar\nu)=0
		\qquad\forall i\in  I^{00}(\bar x).
	\end{equation*}
\end{lemma}
\begin{proof}
	Let us define the function
	\begin{equation*}
		\hat g:[0,1]^p\times \set{0,1}^p\to\nonneg,
		\quad
		(y,\alpha)\mapsto
		\paren[\Bigg]{\prod_{i\in I^p,\,\alpha_i=1}y_i}
		\cdot\paren[\Bigg]{\prod_{i\in I^p,\,\alpha_i=0}(1-y_i)}.
% 		\cdot\prod_{i\in I^p,\,\alpha_i=0}(1-y_i).
	\end{equation*}
	Note that for each $y\in[0,1]^p$
	there exists some $\beta\in\set{0,1}^p$ with $\hat g(y,\beta)>0$.
	Therefore, the normalized function
	\begin{equation*}
		g:[0,1]^p\times \set{0,1}^p\to[0,1],
		\quad
		(y,\alpha)\mapsto
		\frac{\hat g(y,\alpha)}{
		\sum_{\beta\in\set{0,1}^p}\hat g(y,\beta)}
	\end{equation*}
	is well-defined and has the property
	$\sum_{\alpha\in\set{0,1}^p}g(y,\alpha)=1$.
	Note that $\hat g(\cdot,\alpha)$ and $g(\cdot,\alpha)$ 
	are continuous for each $\alpha\in\set{0,1}^p$.
	Our next goal is to apply \cref{thm:poincare_miranda}
	to the function
	\begin{equation*}
		h:[0,1]^p\to\R^p,
		\quad
		y\mapsto 
		\begin{cases}
			\psi_i\paren[\Big]{
				\sum_{\alpha\in\set{0,1}^p}g(y,\alpha)(\mu^\alpha,\nu^\alpha)
			}
			&\text{if}\;i\in  I^{00}(\bar x)
			\\
			0 &\text{if}\;i\in I^p\setminus I^{00}(\bar x).
		\end{cases}
	\end{equation*}
	Clearly, $h$ is continuous.
	Let us verify that the required sign conditions for $h$ hold.
	For $i\in I^p\setminus  I^{00}(\bar x)$, we have $h_i(y)=0$
	and the assumption is satisfied.
	Let $i\in  I^{00}(\bar x)$ and $y\in [0,1]^p$ be given.
	We first consider the case that $y_i=0$ holds.
	For $\alpha\in\set{0,1}^p$ with $\alpha_i=1$ we have
	$\hat g(y,\alpha)=0$ and $g(y,\alpha)=0$.
	We obtain
	\begin{equation*}
		\sum_{\alpha\in\set{0,1}^p}
		g(y,\alpha)\mu^\alpha_i
		= \sum_{\alpha\in\set{0,1}^p, i\in I_{\alpha=0}^{00}(\bar x)} 
		g(y,\alpha)\mu^\alpha_i
		\geq0
	\end{equation*}
	from $(\mu^\alpha,\nu^\alpha)\in A^\alpha$.
	Together with \eqref{eq:psi_sign_condition} this implies
	\begin{equation*}
		h_i(y) 
		= \psi_i\paren[\Big]{\sum_{\alpha\in\set{0,1}^p}
		g(y,\alpha)(\mu^\alpha,\nu^\alpha)}
		\leq0.
	\end{equation*}
	The other case with $y_i=1$ works similarly:
	There, for $\alpha\in\set{0,1}^p$ with $\alpha_i=0$
	we have $\hat g(y,\alpha)=0$ and $g(y,\alpha)=0$.
	We obtain
	\begin{equation*}
		\sum_{\alpha\in\set{0,1}^p}
		g(y,\alpha)\nu^\alpha_i
		= \sum_{\alpha\in\set{0,1}^p, i\in I_{\alpha=1}^{00}(\bar x)} 
		g(y,\alpha)\nu^\alpha_i
		\geq0
	\end{equation*}
	from $(\mu^\alpha,\nu^\alpha)\in A^\alpha$.
	Together with \eqref{eq:psi_sign_condition} this implies
	\begin{equation*}
		h_i(y) 
		= \psi_i\paren[\Big]{\sum_{\alpha\in\set{0,1}^p}
		g(y,\alpha)(\mu^\alpha,\nu^\alpha)}
		\geq0.
	\end{equation*}
	Therefore, \cref{thm:poincare_miranda} can be applied,
	which yields a point $\bar y\in [0,1]^p$ with $h(\bar y)=0$.
	We define
	\begin{equation*}
		(\bar\mu,\bar\nu):=
		\sum_{\alpha\in\set{0,1}^p} g(\bar y,\alpha)(\mu^\alpha,\nu^\alpha).
	\end{equation*}
	Due to the properties of $g$, the point $(\bar\mu,\bar\nu)$
	is indeed a point 
	in the convex hull of 
	$\set{(\mu^\alpha,\nu^\alpha)\given \alpha\in\set{0,1}^p}$.
	For $i\in I^{00}(\bar x)$, we also obtain the remaining condition by
	\begin{align*}
		\psi_i(\bar\mu,\bar\nu)
		= \psi_i\paren[\Big]{\sum_{\alpha\in\set{0,1}^p}
		g(\bar y,\alpha)(\mu^\alpha,\nu^\alpha)}
		= h_i(\bar y) =0.
	\end{align*}
\end{proof}
We mention that any feasible function $\psi$ satisfies
$\psi_i(a,b)=0$ if $a_i\geq0\land b_i\geq0$, which is the area
that corresponds to the system of strong stationarity.
By choosing a suitable function $\psi$,
we can use the previous \lcnamecref{lem:a_implies_psi} to show that 
\pstat\alpha-stationarity is a first-order necessary optimality condition
under MPCC-GCQ.
\begin{theorem}
	\label{thm:between_m_and_s}
	Suppose $\bar x$ is an \astat\alpha-stationary point of \eqref{eq:mpcc}
	for all $\alpha\in\set{0,1}^p$.
	Then $\bar x$ is a \pstat\alpha-stationary point
	for all $\alpha\in\set{0,1}^p$.

	In particular, if $\bar x$ is a local minimizer of \eqref{eq:mpcc}
	that satisfies MPCC-GCQ,
	then it is \pstat\alpha-stationary
	for all $\alpha\in\set{0,1}^p$.
\end{theorem}
\begin{proof}
	Let $\beta\in\set{0,1}^p$ be given.
	We will show that $\bar x$ is \pstat\beta-stationarity.
	For each $\alpha\in\set{0,1}^p$, let
	$(\lambda^\alpha,\eta^\alpha,\mu^\alpha,\nu^\alpha)$
	be multipliers which satisfy the system of \astat{\alpha}-stationarity.
	We want to apply \cref{lem:a_implies_psi}.
	By definition, we have $(\mu^\alpha,\nu^\alpha)\in A^\alpha$,
	where $A^\alpha$ is defined as in \eqref{eq:def_set_A_alpha}.
% 	(with $I^{00}(\bar x)$ instead of $\hat I$).
	We use the function $\psi:\R^p\times\R^p\to\R^p$ which
	is given by
	\begin{equation*}
		\psi_i(a,b):=
		\begin{cases}
			\max(-a_i,\min(0,b_i))
			&\text{if}\; \beta_i=0
			\\
			\min(b_i,\max(0,-a_i))
			&\text{if}\; \beta_i=1
		\end{cases}
	\end{equation*}
	for all $i\in I^p$, $a,b\in\R^p$.
	It can be checked that these functions are continuous and satisfy
	\eqref{eq:psi_sign_condition}.
	Thus, we can apply \cref{lem:a_implies_psi}
	and there exists a convex combination
	$(\bar\lambda,\bar\eta,\bar\mu,\bar\nu)$ of the multipliers
	$(\lambda^\alpha,\eta^\alpha,\mu^\alpha,\nu^\alpha)$
	such that $\psi_i(\bar\mu,\bar\nu)=0$ holds for all 
	$i\in I^{00}(\bar x)$.
	Let us check that $(\bar\mu,\bar\nu)$ satisfy the additional conditions
	for \pstat\beta-stationarity.
	For $i\in I_{\beta=0}^{00}(\bar x)$ we have
	\begin{align*}
		&0 = \psi_i(\bar\mu,\bar\nu)
		= \max(-\bar\mu_i,\min(0,\bar\nu_i))
		\\&\qquad \iff
		(-\bar\mu_i=0\land \min(0,\bar\nu_i)\leq0)
		\lor (\min(0,\bar\nu_i)=0\land -\bar\mu_i\leq0)
		\\ &\qquad \iff
		(\bar\mu_i \geq0 \land\bar\nu_i\geq0) \lor \bar\mu_i =0.
	\end{align*}
	Similarly, for $i\in I_{\beta=1}^{00}(\bar x)$ we have
	\begin{align*}
		&0 = \psi_i(\bar\mu,\bar\nu)
		= \min(\bar\nu_i,\max(0,-\bar\mu_i))
		\\&\qquad \iff
		(\bar\nu_i=0\land \max(0,-\bar\mu_i)\geq0)
		\lor (\max(0,-\bar\mu_i)=0\land\bar\nu_i\geq0)
		\\&\qquad \iff
		(\bar\mu_i \geq0 \land\bar\nu_i\geq0) \lor \bar\nu_i =0.
	\end{align*}
	It remains to show that
	$(\bar\lambda,\bar\eta,\bar\mu,\bar\nu)$
	satisfies the system of weak stationarity.
	This, however, is true due to the convex nature
	of the system of weak stationarity
	and because
	$(\lambda^\alpha,\eta^\alpha,\mu^\alpha,\nu^\alpha)$
	satisfies the system of weak stationarity for all $\alpha\in\set{0,1}^p$.

	Due to \cref{prop:astat}, the \pstat\beta-stationarity condition
	is also satisfied if $\bar x$ is a local minimizer of \eqref{eq:mpcc}
	that satisfies MPCC-GCQ.
\end{proof}
This result will be generalized in \cref{sec:weird_stat}
by considering other choices for $\psi$.
While the \pstat\alpha-stationarity also follows directly from
\cref{thm:weird_stat_mpcc}~\ref{item:cd_stat_mpcc},
we believe it is useful to also present this
simpler proof for the interesting case of stationarity
conditions between strong and M-stationarity.

Some equivalences involving \pstat\alpha-stationarity
are shown in \cref{cor:equiv_mpcc}.

\section{Other new stationarity conditions for MPCCs}
\label{sec:weird_stat}

We can generalize the approach in \cref{sec:between_mpcc}
to obtain more stationarity conditions under MPCC-GCQ.
However, these stationarity conditions do not necessarily
lie between strong and M-stationarity, but only between
strong and C-stationarity.
As a special case, we obtain \cstat{d}-stationarity of local
minimizers under MPCC-GCQ for all $d\in [0,1]^p$.
\begin{lemma}
	\label{lem:combine_astat_weird}
	For each $i\in I^p$,
	let $C_i\subset \nonpossq$ be a closed, connected and unbounded set
	with $0\in C_i$.
	Furthermore, let $\bar x\in\R^n$ be a point
	and for all $\alpha\in\set{0,1}^p$,
	let points $(\mu^\alpha,\nu^\alpha)\in A^\alpha$ be given, 
	where the set $A^\alpha$ is defined in
	\eqref{eq:def_set_A_alpha}.
	Then there exists a point $(\bar\mu,\bar\nu)$
	in the convex hull of 
	$\set{(\mu^\alpha,\nu^\alpha)\given \alpha\in\set{0,1}^p}$
	with
	\begin{equation*}
		(\bar\mu_i\geq0\land\bar\nu_i\geq0)
		\lor (\bar\mu_i,\bar\nu_i)\in C_i
		\qquad\forall i\in  I^{00}(\bar x).
	\end{equation*}
\end{lemma}
\begin{proof}
	Let $i\in  I^{00}(\bar x)$ be fixed.
	We define $\hat C:=C_i\cup \nonnegpow{2}\subset \R^2$,
	which is again a closed and connected set.

	Let $O_1,O_2\subset \R^2$ be the connected components of 
	$\R^2\setminus \hat C$
	with the properties $(-1,1)\in O_1$ and $(1,-1)\in O_2$.
	The sets $O_1$, $O_2$ are open because $\R^2$ is a 
	locally connected space and 
	connected components in a locally connected space are open.
	In order to apply \cref{lem:a_implies_psi},
	we want to construct functions $\psi_i:\R^2\to\R$
	that are $0$ only on $\hat C$ and satisfy some sign conditions.
	This requires that $O_1$ and $O_2$ are different
	connected components.
	We will show that the open sets
	$O_1$ and $O_2$ are different connected components.
	Suppose by contradiction that $O_1=O_2$.
% new proof attempt using Munkres
	\ifmunkres
		We define the closed set $\tilde C:=\R^2\setminus O_1$.
		We claim that $\tilde C$ is again a connected set.
		Let $D_1$ be the connected component of $\tilde C$
		that includes $\hat C$.
		We define $D_2:=\tilde C\setminus D_1$.
		Because $D_1$ is relatively open with respect to $\tilde C$,
		the set $D_2$ is closed (with respect to both $\tilde C$ and $\R^2$).
		Because $O_1$ is relatively closed with respect to 
		$\R^2\setminus\hat C$,
		the set $\R^2\setminus (\hat C\cup O_1)$ is open,
		and therefore the set 
		$D_2=(\R^2\setminus(\hat C\cup O_1))\setminus D_1$
		is open in $\R^2$.
		Since $D_2$ is both open and closed, we have $D_2=\emptyset$,
		and therefore $\tilde C=D_1$ is connected.
		We define the sets $\tilde C^+:=\tilde C\cap \nonnegpow{2}$
		and $\tilde C^-:=\tilde C\cap\nonpospow{2}$.
		Note that $\tilde C=\tilde C^+\cup \tilde C^-$
		and that $\tilde C^+$, $\tilde C^-$ are both 
		closed, connected, and unbounded sets.

		Next, let $S^2\subset\R^3$ denote the unit sphere,
		let $\bar n\in S^2$ be a point, and let 
		$h:\R^2\to S^2\setminus\set{\bar n}$
		denote a homeomorphism.
		Because $\tilde C^+$, $\tilde C^-$ are unbounded,
		the sets $h(\tilde C^+)\cup\set{\bar n}$
		and $h(\tilde C^-)\cup\set{\bar n}$
		are closed and connected.
		Further, note that the intersection of these sets consists
		precisely of the two (distinct) points $\bar n$ and $h(0)$.
		We apply \cite[Theorem~61.4]{Munkres2000},
		which states that the complement $h(O_1)$
		of $S^2\setminus (h(\tilde C^+)\cup h(\tilde C^-)\cup \set{\bar n})$
		is not connected.
		However, since $h$ is a homeomorphism, this is a contradiction
		to the connectedness of $O_1$.
		Hence, our assumption $O_1=O_2$ was wrong and they
		are different connected components.
	\else
% path based proof variant:
		Since $O_1$ is an open and connected set, it is also path-connected.
		Thus, there exists a path from $(-1,1)$ to $(1,-1)$ in $O_1$.
		Let $K_1$ denote the image of the path.
		Since $\hat C$ is closed and $K_1$ is compact,
		there exists a minimum distance $d>0$ of $K_1$ from $\hat C$.
		We define the open set 
		$O_{\hat C}:=\set{y\in\R^2\given \dist(y,\hat C)<d/2}$.
		This set is also connected:
		a nontrivial connected component $G$ of $O_{\hat C}$
		would lead to a nontrivial connected component $G\cap \hat C$
		of $\hat C$.

		Since $C_i$ is unbounded, there exists a point
		$\hat x\in C_i$ with $\hat x_1\le y_1\;\forall y\in K_1$ or
		$\hat x_2\le y_2\;\forall y\in K_1$.
		Without loss of generality we assume that
		$\hat x_1\le y_1\;\forall y\in K_1$ holds
		(otherwise one could just exchange coordinates for 
		the rest of the proof that $O_1\ne O_2$).
		Likewise, there exists a point 
		$(\hat z_1,0)\in \nonnegsq\subset\hat C$
		with $\hat z_1 \ge y_1\;\forall y\in K_1$.
		Since $O_{\hat C}$ is open and connected, we can find a path
		$\gamma^2:[0,1]\to O_{\hat C}$ with
		$\gamma^2(0)=\hat x$ and $\gamma^2(1)=(\hat z_1,0)$.
		We also define 
		$\gamma^1(0):=(-1,2+\sup\set{\gamma^2_2(a)\given a\in[0,1]}) \in O_1$
		and $\gamma^1(1):=
		(1,\inf\set{\gamma^2_2(a)\given a\in[0,1]}-2)\in O_2=O_1$,
		and choose a path $\gamma^1:[0,1]\to O_1$
		that goes from $\gamma^1(0)$ to $(-1,1)$ in a straight line,
		then continuous through $K_1$ to the point $(1,-1)$,
		and goes to $\gamma^1(1)$ in a straight line from there.
		Note that the minimum distance of $\gamma^1([0,1])$ to $\hat C$
		is still equal to $d>0$.

		We define the continuous function $h:[0,1]^2\to\R^2$ via
		$h(y)=\gamma^2(y_1)-\gamma^1(y_2)$.
		Then the conditions of \cref{thm:poincare_miranda}
		are true due to 
		$\gamma^2_1(0)=\hat x_1\leq \gamma^1_1(a)\leq \hat z_1=\gamma^2_1(1)$
		and
		$\gamma^1_2(1)\leq \gamma^2_2(a)\leq \gamma^1_2(0)$
		for all $a\in [0,1]$.
		Thus, there exists a point $\bar y\in[0,1]^2$ with $h(\bar y)=0$,
		i.e.\ $\gamma^2(\bar y_1)=\gamma^1(\bar y_2)$.
		Thus, the paths intersect, which leads to
		to $d=\dist(\gamma^1([0,1]),\hat C) \leq
		\dist(\gamma^1([0,1]),\gamma^2(\bar y_1))
		+\dist(\gamma^2(\bar y_1),\hat C)
		< 0 + d/2$, which is a contradiction.
		Hence, our assumption $O_1=O_2$ was wrong and they
		are different connected components.
	\fi
	Finally, we define
	\begin{equation*}
		\hat\psi_i:\R^2\to\R,
		\quad
		y\mapsto
		\begin{cases}
			\dist(y,\hat C)
			&\text{if}\; y\in O_1
			\\
			-\dist(y,\hat C)
			&\text{if}\; y\not\in O_1
		\end{cases}.
	\end{equation*}
	Because $O_1$ is a connected component of 
	$\R^2\setminus\hat C$, the function $\hat\psi_i$ 
	is continuous.
	Since $O_1$ and $O_2$ are different connected components
	and $O_2$ is open,
	we have $\hat\psi_i(y)<0$ on $O_2$.
	Because of
	$\nonneg\times\R\subset \hat C\cup O_2$
	and
	$\R\times \nonneg\subset \hat C\cup O_1$
	it follows that $\hat\psi_i(y_1,y_2)\leq0$ if $y_1\geq0$
	and $\hat\psi_i(y_1,y_2)\geq0$ if $y_2\geq0$ holds.
	Therefore, the function
	\begin{equation*}
		\psi:\R^p\times\R^p\to\R^p,
		\quad (a,b)\mapsto
		\paren[\big]{ \hat\psi_i(a_i,b_i)}_{i\in I^p}
	\end{equation*}
	satisfies the assumptions of \cref{lem:a_implies_psi}
% 	it follows that the conditions on $\hat\psi_i$
% 	in \cref{lem:a_implies_psi} are satisfied
	and we can apply the \lcnamecref{lem:a_implies_psi}.
	Thus, there exists a point $(\bar\mu,\bar\nu)$
	in the convex hull of 
	$\set{(\mu^\alpha,\nu^\alpha)\given \alpha\in\set{0,1}^p}$
	with
	$\psi_i(\bar\mu,\bar\nu)=0$
% 	$(\bar\mu_i\geq0\lor \bar\nu_i\geq0)\lor \psi_i(\bar\mu,\bar\nu_i)=0$
	for all $i\in I^{00}(\bar x)$.
	The result then follows from the definition
	of $\psi_i$ and $\hat C$,
	in particular the equivalence
	\begin{equation*}
		\psi_i(a,b)=0
		\iff
		\hat\psi_i(a_i,b_i)=0
		\iff
		(a_i,b_i)\in \hat C
		\iff
		(a_i\geq0\land b_i\geq0)\lor (a_i,b_i)\in C_i
	\end{equation*}
	for all $a,b\in\R^p$, $i\in I^{00}(\bar x)$.
\end{proof}

\begin{theorem}
	\label{thm:weird_stat_mpcc}
	Suppose $\bar x$ is an \astat\alpha-stationary point of \eqref{eq:mpcc}
	for all $\alpha\in\set{0,1}^p$.
	Then we have the following conditions.
	% Let $\bar x$ be a local minimizer of \eqref{eq:mpvc}.
	% Suppose that MPVC-GCQ holds at $\bar x$.
% 	Then there exists multipliers
% 	$(\bar\eta,\bar\lambda,\bar\mu,\bar\nu)$
% 	that satisfy the system of weak stationarity.
% 	Additionally, we have the following conditions.
% 	Then we have the following conditions.
	\begin{enumerate}
		\item
			\label{item:weird_stat_mpcc}
			For each $i\in I^p$,
			let $C_i\subset \nonpossq$ be a closed, connected and unbounded set
			with $0\in C_i$.
			Then there exists multipliers
			$(\bar\eta,\bar\lambda,\bar\mu,\bar\nu)$
			that satisfy the system of weak stationarity and
			\begin{equation*}
				(\bar\mu_i\geq0\land\bar\nu_i\geq0)
				\lor (\bar\mu_i,\bar\nu_i)\in C_i
				\qquad\forall i\in I^{00}(\bar x).
			\end{equation*}
		\item
			\label{item:cd_stat_mpcc}
			The point $\bar x$ is \cstat{d}-stationary
			for all $d\in [0,1]^p$.
% 			Let vectors $c,d\in\nonnegpow{p}$ be given.
% 			Then there exists multipliers
% 			$(\bar\eta,\bar\lambda,\bar\mu,\bar\nu)$
% 			that satisfy the system of weak stationarity and
% 			\begin{equation*}
% 				(\bar\mu_i\geq0\land\bar\nu_i\geq0)
% 				\lor c_i\bar\mu_i=d_i\bar\nu_i
% 				\qquad\forall i\in I^{00}(\bar x).
% 			\end{equation*}
	\end{enumerate}
	In particular, these stationarity conditions
	are satisfied for local minimizers $\bar x$
	if MPCC-GCQ holds at $\bar x$.
\end{theorem}
\begin{proof}
	We start with part~\ref{item:weird_stat_mpcc}.
	For each $\alpha\in\set{0,1}^p$, let
	$(\lambda^\alpha,\eta^\alpha,\mu^\alpha,\nu^\alpha)$
	be multipliers which satisfy the system of \astat{\alpha}-stationarity.
	We want to apply \cref{lem:combine_astat_weird}.
% 	with $\hat I=I^{00}(\bar x)$.
	By definition, we have $(\mu^\alpha,\nu^\alpha)\in A^\alpha$,
	where $A^\alpha$ is defined as in \eqref{eq:def_set_A_alpha}.
	Thus, an application of \cref{lem:combine_astat_weird}
	yields a convex combination
	$(\bar\lambda,\bar\eta,\bar\mu,\bar\nu)$ of the multipliers
	$(\lambda^\alpha,\eta^\alpha,\mu^\alpha,\nu^\alpha)$
	which satisfies $ (\bar\mu_i\geq0\land\bar\nu_i\geq0)
	\lor (\bar\mu_i,\bar\nu_i)\in C_i$ for all $i\in I^{00}(\bar x)$.
	Furthermore,
	$(\bar\lambda,\bar\eta,\bar\mu,\bar\nu)$
	satisfies the system of weak stationarity
	due to the convex nature of the weak stationarity system
	and because
	$(\lambda^\alpha,\eta^\alpha,\mu^\alpha,\nu^\alpha)$
	satisfy the system of weak stationarity for all $\alpha\in\set{0,1}^p$.

	Part~\ref{item:cd_stat_mpcc} follows from
	part~\ref{item:weird_stat_mpcc} with the choice
	$C_i:=\set{(a,b)\in \nonpossq\given (1-d_i)a = d_ib}$,
	which is indeed a closed, connected and unbounded set with $0\in C_i$.

	Due to \cref{prop:astat},
	these stationarity conditions are also 
	satisfied if $\bar x$ is a local minimizer of
	\eqref{eq:mpcc} that satisfies MPCC-GCQ.
\end{proof}
Clearly, part~\ref{item:cd_stat_mpcc} was only a special case
of part~\ref{item:weird_stat_mpcc} in \cref{thm:weird_stat_mpcc},
but we included it because it is a more natural condition.

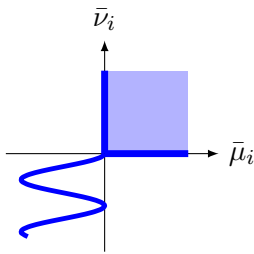
\begin{figure}[t]
% 	\centering
		\def\weirdstat{1}
		\centering
\begin{tikzpicture}[scale=1.1]
	\fill[blue, opacity=0.3]
	(1,0) -- (1,1) --(0,1) --(0,0) --cycle;
	\ifthenelse{\castat=1}{
		\fill[blue, opacity=0.3]
		(-1,0) -- (-1,-1) --(0,-1) --(0,0) --cycle;
	}{}
	\ifthenelse{\aastat=1}{
		\fill[blue, opacity=0.3]
		(1,0) -- (1,-1) --(0,-1) --(0,0) --cycle;
	}{}
	\ifthenelse{\abstat=1}{
		\fill[blue, opacity=0.3]
		(-1,0) -- (-1,1) --(0,1) --(0,0) --cycle;
	}{}
	\draw[axis] (-1,0) -- (1,0)
% 	node[right=2* \nudge cm] {$G(x)$\sometimes{, $\mu$}};
	node[right=2* \nudge cm] {$\bar\mu_i$};
	\draw[axis] (0,-1) -- (0,1)
% 	node[above=2*\nudge cm] {$H(x)$\sometimes{, $\nu$}};
	node[above=2*\nudge cm] {$\bar\nu_i$};
	\draw[line width=\feaslw,blue] (0,0) -- (1,0) ;
	\draw[line width=\feaslw,blue] (0,0) -- (0,1) ;
	\ifthenelse{\pastat=1 \OR \aastat=1 \OR \castat=1}{
		\draw[line width=\feaslw,blue] (0,-1) -- (0,0) ;
	}{}
	\ifthenelse{\cdstat=1}{
		\draw[line width=\feaslw,blue] (0,0) -- (-1,-1) ;
	}{}
	\ifthenelse{\weirdstat=1}{
		\draw[line width=0.7*\feaslw,smooth,blue,domain=0:1,variable=\x ] 
% 		plot ({-sin(300*\x)*sin(300*\x)}, {-\x});
		plot ({-sin(5*deg(\x))*sin(5*deg(\x))}, {-\x});
	}{}
	\ifthenelse{\pbstat=1 \OR \abstat=1 \OR \castat=1}{
		\draw[line width=\feaslw,blue] (-1,0) -- (0,0) ;
	}{}
\end{tikzpicture}
\def\pastat{0}
\def\pbstat{0}
\def\aastat{0}
\def\abstat{0}
\def\castat{0}
\def\cdstat{0}
	\caption{geometric illustration of the stationarity condition
		of \cref{thm:weird_stat_mpcc}~\ref{item:weird_stat_mpcc}
	with $i\in I^{00}(\bar x)$, 
	$C_i=\set{(-\sin^2(5a),a)\mid a\leq0}$
	}
	\label{fig:weird_stat_mpcc}
\end{figure}

As a simple \lcnamecref{cor:equiv_mpcc}, let us state
some relations among the new stationarity conditions.
\begin{corollary}
	\label{cor:equiv_mpcc}
	Let $\bar x$ be a feasible point of \eqref{eq:mpcc}.
	The following are equivalent.
	\begin{enumerate}
		\item
			\label{item:cor:bstat}
			$\bar x$ is linearized-B-stationary,
		\item
			\label{item:cor:astat}
			$\bar x$ is \astat\alpha-stationary for all $\alpha\in\set{0,1}^p$,
		\item
			\label{item:cor:cdstat}
			$\bar x$ is \cstat{d}-stationary for all $d\in [0,1]^p$,
		\item
			\label{item:cor:pstat}
			$\bar x$ is \pstat\alpha-stationary for all $\alpha\in\set{0,1}^p$.
	\end{enumerate}
\end{corollary}
\begin{proof}
	The equivalence 
	\ref{item:cor:bstat}$\iff$\ref{item:cor:astat}
	was already stated in
	\cref{prop:astat}~\ref{item:bstat_iff_astat},
	and \ref{item:cor:astat}$\implies$\ref{item:cor:cdstat}
	was shown in \cref{thm:weird_stat_mpcc}~\ref{item:cd_stat_mpcc}.
	Finally, the implication
	\ref{item:cor:cdstat}$\implies$\ref{item:cor:pstat}
	is trivial
	and \ref{item:cor:pstat}$\implies$\ref{item:cor:astat}
	follows from
	\cref{cor:relations}~\ref{item:pstat_implies_astat}.
\end{proof}

\section{A simple proof of M-stationarity for MPVCs}
\label{sec:simple_mstat}

We turn our attention to MPVCs.
As one can see in
\cref{fig:mstat_mpvc,fig:astat0_mpvc},
% if $i\in I^{00}(\bar x)$,
% if $i\in I_{\alpha=0}^{00}(\bar x)$,
the feasible set for the multipliers $(\bar\mu_i,\bar\nu_i)$ 
for \astat0-stationarity is a subset of
the feasible set for M-stationarity.
In particular, \astat0-stationarity implies
M-stationarity.
Note that such a relation does not hold for MPCCs.
Since \astat\alpha-stationarity is usually easy to show under MPVC-GCQ
(see also \cref{prop:astat_mpvc}), this will lead to a very
simple and short proof of M-stationarity under MPVC-GCQ.
This is a known result, see, e.g.\
\cite[Theorem~3.4]{HoheiselKanzow2008}, but our proof is much simpler.
Furthermore, 
GCQ for \paramref{eq:nlp_mpvc}{0}
also works as a constraint qualification for M-stationarity.
\begin{theorem}
	\label{thm:simple_m_stat}
	Let $\bar x$ be a minimizer of \eqref{eq:mpvc}.
	Suppose that
	\begin{equation}
		\label{eq:or_cq}
		\text{MPVC-GCQ}\quad\text{or}\quad
		\text{GCQ for \paramref{eq:nlp_mpvc}{0}}
	\end{equation}
	holds at $\bar x$.
	Then $\bar x$ is \astat0-stationary.
	In particular,
	$\bar x$ is an M-stationary point.
\end{theorem}
\begin{proof}
	We have $-\nabla f(\bar x)\in\tmpvc(\bar x)\polar$
	due to \cref{lem:basic_nlp_theory}~\ref{item:b_stat}.
	In the case that MPVC-GCQ holds, we can use
	$\tnlplin{0}(\bar x)\subset\tmpvclin(\bar x)$
	to obtain
	$\tmpvc(\bar x)\polar = \tmpvclin(\bar x)\polar
	\subset \tnlplin{0}(\bar x)\polar$.
	In the case that GCQ holds for
	\paramref{eq:nlp_mpvc}{0}, we can use
	$\tnlp0(\bar x)\subset\tmpvc(\bar x)$
	to obtain
	$\tmpvc(\bar x)\polar \subset \tnlp0(\bar x)\polar
	= \tnlplin{0}(\bar x)\polar$.
	In both cases we have $-\nabla f(\bar x)\in \tnlplin0(\bar x)\polar$.
	By \cref{lem:basic_nlp_theory}~\ref{item:lin_implies_kkt},
% 	$\bar x$ is a KKT point of 
% 	\paramref{eq:nlp_mpvc}{0}, 
% 	i.e.\ 
	$\bar x$ is \astat0-stationary.

% 	The M-stationarity of $\bar x$ 
% 	follows directly from the \astat0-stationarity,
	If 
	$(\bar\lambda,\bar\eta,\bar\mu,\bar\nu)$
	satisfies the system of \astat0-stationarity,
	then we have
	$\bar\mu_i=0$ for all $i\in I^{00}(\bar x)$ and therefore
	the multipliers also satisfy the system of M-stationarity.
\end{proof}
Note that this proof does not rely
on the more complicated methods from 
\cref{sec:between_mpcc,sec:weird_stat}
or on advanced techniques from variational analysis
such as the limiting normal cone.
And because \eqref{eq:or_cq} is a weaker condition
than MPVC-GCQ (see \cref{ex:mpvcgcq_not_gcq_emptyset} below), 
we have even generalized the result
of M-stationarity under MPVC-GCQ slightly.
We mention that a similarly elementary method
was used to show M-stationarity for
\emph{mathematical programs with switching constraints} (MPSCs)
in \cite[Theorem~5.1]{Mehlitz2019:1}.
% Also, observe that \astat0-stationarity is the same as
% \pstat0-stationarity for MPVCS

One might wonder what the relationship is between
MPVC-GCQ and GCQ for 
\paramref{eq:nlp_mpvc}{0}.
The following two counterexamples show that
neither implies the other.
\begin{example}
	\label{ex:mpvcgcq_not_gcq_emptyset}
	We consider the setting with $n=2$, $p=l=m=1$,
	$G(x)=x_1$, $H(x)=x_2$, and $\bar x=(0,0)$.
	\begin{enumerate}
		\item
			\label{item:not_gcq_emptyset_mpvcgcq}
			If $h(x)=x_2^2-x_1$ and $g(x)=0$,
			then MPVC-GCQ does not hold at $\bar x$,
			but GCQ holds for \paramref{eq:nlp_mpvc}{0}
			at $\bar x$.
		\item
			\label{item:gcq_emptyset_not_mpvcgcq}
			If $h(x)=x_1^2-x_2$, and $g(x)=x_1$.
			Then MPVC-GCQ holds at $\bar x$, but GCQ does not hold
			for \paramref{eq:nlp_mpvc}{0} at $\bar x$.
	\end{enumerate}
\end{example}
\begin{proof}
	We have $I^p=I^{00}(\bar x)=\set1$.
	For part~\ref{item:not_gcq_emptyset_mpvcgcq},
	the point $\bar x=(0,0)$ is the only feasible point of
	\eqref{eq:mpvc} and \paramref{eq:nlp_mpvc}{0}.
	Thus, we have $\tmpvc(\bar x)=\tnlp0=\set{(0,0)}$.
	One can also calculate
	$\tmpvclin(\bar x)=\set0\times\nonneg$
	and $\tnlplin0(\bar x)=\set{(0,0)}$.
	Taking polar cones yields the claim.

	For part~\ref{item:gcq_emptyset_not_mpvcgcq},
	the feasible set of \paramref{eq:nlp_mpvc}0
	is again just $\set{\bar x}$, but the feasible set
	of \eqref{eq:mpvc} is $\set{x\in\R^2\given x_1^2=x_2, x_1\leq0}$.
	Thus, we have $\tmpvc(\bar x)=\nonpos\times\set0$
	and $\tnlp0=\set{(0,0)}$.
	For the linearization cones, one can calculate
	$\tmpvclin(\bar x)=\tnlp0(\bar x)=\nonpos\times\set0$.
	Taking polar cones yields the claim.
\end{proof}

\section{New stationarity conditions for MPVCs}
\label{sec:between_mpvc}

In this \lcnamecref{sec:between_mpvc} we will use
the results of \cref{sec:between_mpcc,sec:weird_stat}
and apply them in the setting of MPVCs.
This will lead to new stationarity conditions for MPVCs.
An important result will be that \pstat\alpha-stationarity
is a first-order necessary optimality condition under MPVC-GCQ
for all $\alpha\in\set{0,1}^p$.

% We start with a \lcnamecref{prop:astat_mpvc}
% that local minimizers are \astat\alpha-stationary
% under MPVC-GCQ for all $\alpha\in\set{0,1}^p$.
We start with an analogue to \cref{prop:astat} for MPVCs.
\begin{proposition}
	\label{prop:astat_mpvc}
	Let $\bar x\in\R^n$ be a feasible point of \eqref{eq:mpvc}.
	\begin{enumerate}
		\item
			\label{item:bstat_mpvc}
			If $\bar x$ is a local minimizer of \eqref{eq:mpvc}
			that satisfies MPVC-GCQ, then $\bar x$ is
			linearized B-stationary.
		\item
			\label{item:bstat_iff_astat_mpvc}
			The point $\bar x$ is linearized B-stationary if and only if it is
			\astat\alpha-stationary for all $\alpha\in\set{0,1}^p$.
	\end{enumerate}
	In particular, if $\bar x$ is a local minimizer of \eqref{eq:mpvc}
	that satisfies MPVC-GCQ,
	then it is \astat\alpha-stationary
	for all $\alpha\in\set{0,1}^p$.
\end{proposition}
\begin{proof}
	For part~\ref{item:bstat_mpvc}, we obtain
	$-\nabla f(\bar x)\in \tmpvc(\bar x)\polar
	=\tmpvclin(\bar x)\polar$
	from \cref{lem:basic_nlp_theory}~\ref{item:b_stat} and MPVC-GCQ.
	Thus, $\bar x$ is linearized B-stationary.

% 	the equality
% 	\begin{equation*}
% 		\tmpvclin(\bar x)
% 		=\bigcup_{\alpha\in\set{0,1}^p} \tnlplin{\alpha}(\bar x)
% 	\end{equation*}
% 	for the MPVC case can be found in \cite{}
% 	and the rest of the argument is the same as in \cref{prop:astat}.
	For part~\ref{item:bstat_iff_astat_mpvc},
	we first observe the equality
	\begin{equation*}
		\tmpvclin(\bar x)
		=\bigcup_{\alpha\in\set{0,1}^p} \tnlplin{\alpha}(\bar x),
	\end{equation*}
	which can be shown by direct calculations
	or obtained from \cite[Lemma~2.4]{HoheiselKanzow2008}.
	Therefore, linearized B-stationarity can be written as
	\begin{equation*}
		-\nabla f(\bar x)\in \tmpvclin(\bar x)\polar
		= \paren[\bigg]{ 
		\bigcup_{\alpha\in\set{0,1}^p} \tnlplin{\alpha}(\bar x)}\polar
		= \bigcap_{\alpha\in\set{0,1}^p}\tnlplin{\alpha}(\bar x)\polar.
	\end{equation*}
	Thus, $\bar x$ is linearized B-stationary if and only if
	$-\nabla f(\bar x)\in \tnlplin{\alpha}(\bar x)\polar$
	for all $\alpha\in\set{0,1}^p$.
	However, the latter condition is equivalent to
	\astat\alpha-stationarity of $\bar x$
	by \cref{lem:basic_nlp_theory}~\ref{item:lin_implies_kkt}.
\end{proof}
Now we come to the main result for MPVCs, 
which is an analogue of \cref{thm:between_m_and_s,thm:weird_stat_mpcc}.
The proof is not very difficult, since it is
possible to apply \cref{lem:combine_astat_weird}.
\begin{theorem}
	\label{thm:mpvc_new_stat}
	Suppose $\bar x$ is an \astat\alpha-stationary point of \eqref{eq:mpvc}
	for all $\alpha\in\set{0,1}^p$.
	Then we have the following conditions.
	% Let $\bar x$ be a local minimizer of \eqref{eq:mpvc}.
	% Suppose that MPVC-GCQ holds at $\bar x$.
% 	Then there exists multipliers
% 	$(\bar\eta,\bar\lambda,\bar\mu,\bar\nu)$
% 	that satisfy the system of weak stationarity.
% 	Additionally, we have the following conditions.
	\begin{enumerate}
		\item
			\label{item:mpvc_weird_stat}
			For each $i\in I^p$,
			let $C_i\subset \nonpossq$ be a closed, connected and unbounded set
			with $0\in C_i$.
			Then there exists multipliers
			$(\bar\eta,\bar\lambda,\bar\mu,\bar\nu)$
			that satisfy the system of weak stationarity and
			\begin{equation*}
				% \label{eq:weird_stat}
				(\bar\mu_i=0\land\bar\nu_i\geq0)
				\lor (-\bar\mu_i,\bar\nu_i)\in C_i
				\qquad\forall i\in I^{00}(\bar x).
			\end{equation*}
		\item
			\label{item:mpvc_vector_stat}
			The point $\bar x$ is \cstat{d}-stationary
			for all $d\in [0,1]^p$.
% 			Let vectors $c,d\in\nonnegpow{p}$ be given.
% 			Then there exists multipliers
% 			$(\bar\eta,\bar\lambda,\bar\mu,\bar\nu)$
% 			that satisfy the system of weak stationarity and
% 			\begin{equation*}
% 				(\bar\mu_i=0\land\bar\nu_i\geq0)
% 				\lor -c_i\bar\mu_i = d_i\bar\nu_i
% 				\qquad\forall i\in I^{00}(\bar x).
% 			\end{equation*}
		\item
			\label{item:mpvc_new_stat}
			The point $\bar x$ is \pstat\alpha-stationary
			for all $\alpha\in\set{0,1}^p$.
	\end{enumerate}
	In particular, these stationarity conditions
	are satisfied for local minimizers $\bar x$
	if MPVC-GCQ holds at $\bar x$.
\end{theorem}
\begin{proof}
	We start with part~\ref{item:mpvc_weird_stat}.
	For each $\alpha\in\set{0,1}^p$, let
	$(\lambda^\alpha,\eta^\alpha,\mu^\alpha,\nu^\alpha)$
% 	$\lambda^\alpha\in\R^l$, $\eta^\alpha\in\R^m$,
% 	$\mu^\alpha,\nu^\alpha\in\R^p$
	be multipliers which satisfy the system of \astat{\alpha}-stationarity.
	
	We want to apply \cref{lem:combine_astat_weird}
	with $-\mu$ instead of $\mu$.
	Indeed, we have $(-\mu^\alpha,\nu^\alpha)\in A^\alpha$
	due to the \astat\alpha-stationarity.
	Thus, \cref{lem:combine_astat_weird}
	can be applied and
	there exists a convex combination
	$(\bar\lambda,\bar\eta,\bar\mu,\bar\nu)$
	of the multipliers
	$(\lambda^\alpha,\eta^\alpha,\mu^\alpha,\nu^\alpha)$
	such that
	$
	(-\bar\mu_i\geq0\land\bar\nu_i\geq0)
	\lor (-\bar\mu_i,\bar\nu_i)\in C_i
	$
	for all $i\in I^{00}(\bar x)$ holds.
	Because $\mu^\alpha_i\geq0$ holds for all $i\in I^{00}(\bar x)$
	and $\alpha\in\set{0,1}^p$, the same holds for $\bar\mu_i$.
	Thus, the condition $-\bar\mu_i\geq0$ can be equivalently
	replaced by $\bar\mu_i=0$ if $i\in I^{00}(\bar x)$.
	It remains to show that
	$(\bar\lambda,\bar\eta,\bar\mu,\bar\nu)$
	satisfies the system of weak stationarity.
	This, however, is true due to the convex nature
	of the system of weak stationarity.

	Part~\ref{item:mpvc_vector_stat} follows from
	part~\ref{item:mpvc_weird_stat} with the choice
	$C_i:=\set{(a,b)\in \nonpossq\given (1-d_i)a = d_ib}$,
	which is indeed a closed, connected and unbounded set with $0\in C_i$.

	Part~\ref{item:mpvc_new_stat} then follows trivially from
	part~\ref{item:mpvc_vector_stat}.

	Due to \cref{prop:astat_mpvc},
	these stationarity conditions are also 
	satisfied if $\bar x$ is a local minimizer of
	\eqref{eq:mpvc} that satisfies MPVC-GCQ.
\end{proof}

\begin{figure}[t]
		\def\weirdstat{1}
		\centering
\begin{tikzpicture}[scale=1.1]
% 	\fill[blue, opacity=0.3]
% 	(1,0) -- (1,1) --(0,1) --(0,0) --cycle;
	\ifthenelse{\castat=1}{
		\fill[blue, opacity=0.3]
		(1,0) -- (1,-1) --(0,-1) --(0,0) --cycle;
	}{}
% 	\ifthenelse{\aastat=1}{
% 		\fill[blue, opacity=0.3]
% 		(1,0) -- (1,-1) --(0,-1) --(0,0) --cycle;
% 	}{}
	\ifthenelse{\abstat=1}{
		\fill[blue, opacity=0.3]
		(1,0) -- (1,1) --(0,1) --(0,0) --cycle;
	}{}
	\draw[axis] (-1,0) -- (1,0)
% 	node[right=2* \nudge cm] {$G(x)$\sometimes{, $\mu$}};
	node[right=2* \nudge cm] {$\bar\mu_i$};
	\draw[axis] (0,-1) -- (0,1)
% 	node[above=2*\nudge cm] {$H(x)$\sometimes{, $\nu$}};
	node[above=2*\nudge cm] {$\bar\nu_i$};
% 	\draw[line width=\feaslw,blue] (0,0) -- (1,0) ;
	\draw[line width=\feaslw,blue] (0,0) -- (0,1) ;
	\ifthenelse{\pastat=1 \OR \aastat=1 \OR \castat=1}{
		\draw[line width=\feaslw,blue] (0,-1) -- (0,0) ;
	}{}
	\ifthenelse{\cdstat=1}{
		\draw[line width=\feaslw,blue] (0,0) -- (1,-1) ;
	}{}
	\ifthenelse{\weirdstat=1}{
		\draw[line width=0.7*\feaslw,blue,domain=0:1,variable=\x ] 
		(0,0) -- plot ({(\x)}, {-\x*\x});
% 		(0,0) -- plot ({sqrt(\x)-\x/3}, {-\x});
	}{}
	\ifthenelse{\pbstat=1 \OR \abstat=1 \OR \castat=1}{
		\draw[line width=\feaslw,blue] (1,0) -- (0,0) ;
	}{}
\end{tikzpicture}
\def\pastat{0}
\def\pbstat{0}
\def\aastat{0}
\def\abstat{0}
\def\castat{0}
\def\cdstat{0}
\def\weirdstat{0}
	\caption{geometric illustration of the stationarity condition
		of \cref{thm:mpvc_new_stat}~\ref{item:mpvc_weird_stat}
	with $i\in I^{00}(\bar x)$, 
	$C_i=\set{(a,-a^2)\mid a\leq0}$
	}
	\label{fig:weird_stat_mpvc}
\end{figure}
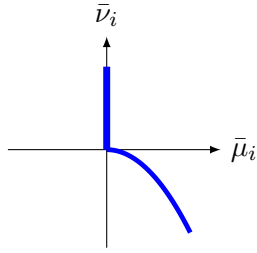

As a corollary, we obtain
an analogue to \cref{cor:equiv_mpcc}.
\begin{corollary}
	\label{cor:equiv_mpvc}
	Let $\bar x$ be a feasible point of \eqref{eq:mpcc}.
	The following are equivalent.
	\begin{enumerate}
		\item
			\label{item:cor:bstat_mpvc}
			$\bar x$ is linearized-B-stationary,
		\item
			\label{item:cor:astat_mpvc}
			$\bar x$ is \astat\alpha-stationary for all $\alpha\in\set{0,1}^p$,
		\item
			\label{item:cor:cdstat_mpvc}
			$\bar x$ is \cstat{d}-stationary for all $d\in [0,1]^p$,
		\item
			\label{item:cor:pstat_mpvc}
			$\bar x$ is \pstat\alpha-stationary for all $\alpha\in\set{0,1}^p$.
	\end{enumerate}
\end{corollary}
\begin{proof}
	The equivalence 
	\ref{item:cor:bstat_mpvc}$\iff$\ref{item:cor:astat_mpvc}
	was already stated in
	\cref{prop:astat_mpvc}~\ref{item:bstat_iff_astat_mpvc},
	and \ref{item:cor:astat_mpvc}$\implies$\ref{item:cor:cdstat_mpvc}
	was shown in \cref{thm:mpvc_new_stat}~\ref{item:mpvc_vector_stat}.
	Finally, the implication
	\ref{item:cor:cdstat_mpvc}$\implies$\ref{item:cor:pstat_mpvc}
	is obvious
	and \ref{item:cor:pstat_mpvc}$\implies$\ref{item:cor:astat_mpvc}
	follows from
	\cref{cor:relations}~\ref{item:pstat_implies_astat}.
\end{proof}

\section{Conclusion and outlook}
\label{sec:conclusion}

We introduced new first-order necessary stationarity conditions
for MPCCs and MPVCs.
In particular, we were able to introduce a new stationarity
condition which lies strictly between strong and M-stationarity.
We also provided a simple, elementary, and short proof of M-stationarity
for MPVCs in \cref{sec:simple_mstat}.

In the future, it might be interesting to investigate to what extend
the methods from this article
can be generalized to mathematical programs with disjunctive constraints
(MPDCs), which is a problem class more general than MPCCs or MPVCs.

Our proof of \pstat\alpha-stationarity 
(and the other new stationarity conditions)
was based on the Poincaré--Miranda theorem, which is not constructive.
Thus, it would be interesting to know whether a more constructive proof
or other alternative proofs can be found.

It is unclear whether the ideas from this article can be used
in Lebesgue or Sobolev spaces.
Here, a problem might be that infinite-dimensional generalizations
of the Poincaré--Miranda theorem or the Brouwer fixed-point theorem
usually require compactness of the set or the function.

Let us describe another open question.
For $i\in I^p$, consider sets $D_i\subset \R^2$.
Can we characterize the sets $D_i$ such that
weak stationarity and 
$(\bar\mu_i,\bar\nu_i)\in D_i\;\forall i\in I^{00}(\bar x)$
is a stationarity conditions (under MPCC-GCQ or MPVC-GCQ)?
Using \cref{thm:weird_stat_mpcc}~\ref{item:weird_stat_mpcc}
and \cref{thm:mpvc_new_stat}~\ref{item:mpvc_weird_stat},
one can describe already a large variety of such sets $D_i$,
% In the present article we described several such sets $D_i$,
but it is not clear whether these are all possibilities
for stationarity conditions described by a set $D_i$.

% \paragraph*{Acknowledgments}

%%fakesection: Bibliography
\printbibliography
\end{document}